\documentclass[journal]{IEEEtran}
\usepackage[T1]{fontenc}

\usepackage{amsmath}
\usepackage{amssymb}
\usepackage{amsthm}
\usepackage{amsmath,amssymb,amsfonts,amsthm,bm,mathtools}
\usepackage{arydshln}
\usepackage{nicematrix}
\usepackage{array}
\usepackage{cite}
\usepackage{balance}
\usepackage{nameref}
\usepackage{caption}
\usepackage{subcaption}

\usepackage{graphicx}
\usepackage[shortlabels]{enumitem}
\usepackage[utf8]{inputenc}
\usepackage[scr]{rsfso}
\usepackage{nicematrix}
\usepackage{arydshln}
\usepackage{comment}
\usepackage[safe]{tipa}

\usepackage{needspace}
\renewcommand{\Needspace}[1]{}

\usepackage{nccmath}
\usepackage{microtype}
\usepackage[strict]{changepage}

\usepackage{lipsum}

\newcommand{\R}{\mathbb{R}}
\newcommand{\N}{\mathbb{N}}

\DeclareMathOperator{\Span}{span}
\DeclareMathOperator{\dist}{dist}
\DeclareMathOperator{\col}{col}
\DeclareMathOperator{\supp}{supp}

\newcommand{\res}{\mathcal R} 
\newtheorem{assumption}{Assumption}[section]

\newtheorem{definition}[assumption]{Definition}
\newtheorem{theorem}[assumption]{Theorem}
\newtheorem{proposition}[assumption]{Proposition}
\newtheorem{lemma}[assumption]{Lemma}
\newtheorem{corollary}[assumption]{Corollary}
\theoremstyle{remark}
\newtheorem{remark}[assumption]{Remark}

\expandafter\def\expandafter\normalsize\expandafter{%
    \normalsize%
    \setlength\abovedisplayskip{4pt}%4
    \setlength\belowdisplayskip{4pt}%4
    \setlength\abovedisplayshortskip{2pt}%3
    \setlength\belowdisplayshortskip{2pt}%3
}

\allowdisplaybreaks
\usepackage{hyperref}   % to set up hyperlinks

\hypersetup{
	colorlinks=true,
	linkcolor=black,
	citecolor=black,
	urlcolor=black,
}

\begin{document}
\raggedbottom % Keep normal spacing within columns instead of stretching it.
\title{Model-Based Koopman--Galerkin Lifting for Guaranteed H-Infinity Output Feedback Control of Nonlinear Systems}

\author{Burak~K\"{u}rk\c{c}\"{u}, Christopher Phan,  Aleksandar Zecevic, Maryam Khanbaghi% 
\thanks{Burak~K\"{u}rk\c{c}\"{u}, Christopher Phan, Aleksandar Zecevic, and Maryam Khanbaghi are with the Department of Electrical and Computer Engineering,
        Santa Clara University, Santa Clara, CA 95053, USA
        (e-mail: \{bkurkcu,cphan,azecevic,mkhanbaghi\}@scu.edu)}% 

}
\markboth{}{}

\maketitle

\begin{abstract}
In this paper, we present a model-based Koopman--Galerkin framework for output-feedback control of nonlinear systems with known dynamics. With a suitable actuator augmentation, the finite-dimensional realization preserves the actuator dynamics and the constant input matrix. The approach is based on an exact decomposition of the lifted dynamics into an LTI system and a finite-order closure residual. It also applies to non-control-affine systems and allows dictionaries that do not span an invariant subspace. Regional bounds are derived for the closure residual, actuator-realization defect, and output-reconstruction error. We use these bounds to formulate a standard $\mathcal{H}_\infty$ output-feedback problem. The controller is linear and dynamic and is implemented using only the measured tracking error. Neither full-state measurements nor a separate state observer are required. Lifting is performed offline. With the state coordinates retained in the dictionary, we obtain explicit transient bounds on the original state and tracking error. A sufficient condition for regional containment is also given. Under this condition, the nonlinear closed loop is forward complete, and the state and tracking error are uniformly ultimately bounded. The cart--pendulum example illustrates the nonlinear closed-loop guarantees. The proposed output feedback gives lower RMSE and peak tracking error on average than Jacobian-based $\mathcal{H}_\infty$ feedback, full-state backstepping, and Koopman-based hybrid MPC.
\end{abstract}

\begin{IEEEkeywords}
Galerkin approximation, $\mathcal{H}_\infty$ control, nonlinear
systems, Koopman theory, output feedback, ultimate boundedness.
\end{IEEEkeywords}

\IEEEpeerreviewmaketitle
\section{Introduction}

A common approach to controlling nonlinear systems involves
designing controllers based on finite-dimensional linear models.
Doyle's work on linear--quadratic--Gaussian (LQG) regulators
offers a historical lesson:  nominal stability and even optimality do
not by themselves imply guaranteed margins~\cite{doyle1978guaranteed}. A related issue arises in lifting-based
control: stabilizing an approximate model does not by itself
guarantee stability of the original nonlinear system.
Establishing such guarantees requires accounting for the
approximation error in the feedback loop~\cite{StrasserARC2026}.
For known continuous-time nonlinear systems, we incorporate
finite-order Galerkin residual bounds into classical
$\mathcal{H}_\infty$ output-feedback synthesis and establish
guarantees for the original nonlinear system.

The classical linearization-based control does offer local
stability but does not specify the neighborhood in which the
stability guarantees hold~\cite{khalil2002nonlinear,ozccelik2023experimental}. Switching between robust controllers designed from local linearizations has also been used for nonlinear flight dynamics~\cite{albayrak2022switching}. The nonlinear terms neglected in the linearization can greatly alter the region of attraction even when the linearized dynamics remain unchanged~\cite{chakraborty2011nonlinear}. Some studies address this limitation
by estimating regions of attraction for linear controllers through
nonlinear analysis~\cite{tedrake2010lqr,chakraborty2011nonlinear}. However, this kind of post-design verification does not deal with the question of how approximation-error bounds over a specified operating region can be included in controller synthesis.

There are other linearization approaches that have been studied in the literature. Examples include Carleman linearization~\cite{KowalskiSteeb1991}, immersion techniques~\cite{LevineMarino1986}, and Koopman operator (KO) theory~\cite{koopman1931hamiltonian,Mezic2005,OttoRowley2021}. Lifting techniques can reduce a nonlinear modeling problem to a modeling problem involving observables. For projection-based lifting, an approximation space and a projection rule are required, as in classical Proper Orthogonal Decomposition (POD)--Galerkin model reduction~\cite{kunisch2002galerkin,Kasnakoglu2009}. Data-driven lifting methods for finite-dimensional approximation include extended dynamic mode decomposition (EDMD)~\cite{Williams2015} and generator techniques in continuous-time~\cite{Klus2020}. For a fixed dictionary, the infinite-data limit of EDMD is the orthogonal Galerkin projection of the discrete-time Koopman operator onto the dictionary span~\cite{KordaMezic2018JNS}. 

On the other hand, when the governing dynamics are given, Koopman direct encoding~\cite{Asada2023Direct} calculates the Galerkin inner products without using trajectory data for KO estimation. However, exact finite-dimensional Koopman closure still requires an invariant subspace of observables, and linear immersions are subject to other limitations imposed by the nonlinear dynamics, such as the nonexistence of continuous one-to-one immersions over certain regions with multiple $\omega$-limit sets~\cite{LiuOzaySontag2025}.

\textit{Control} introduces another fundamental challenge since lifted input operators can depend on the state or even on the input itself~\cite{Iacob2024}. Some approaches estimate both the lifted dynamics and the input matrix from data~\cite{KordaMezic2018Aut}, while control-affine systems can be analyzed using bilinear representations with input-dependent state matrices~\cite{GoswamiPaley2022,PeitzOttoRowley2020}. Obtaining an exact linear time-invariant(LTI) model instead of a bilinear model requires a constant input matrix, which can impose strong constraints on the nonlinear dynamics~\cite{Heeg2026}. Under additional assumptions, exact LTI Koopman representations also require affine linear dynamics and observables. On the actuation side, an approach is to augment the actuator state to preserve the physical input, as in the discrete-time control-coherent Koopman modeling~\cite{Asada2024}.

Regardless of representation, a finite-dimensional lifted model is generally an approximation of the underlying nonlinear model. Residual dynamic mode decomposition (ResDMD)~\cite{ColbrookTownsend2024} uses residual information to verify Koopman spectral computations. For closed-loop control, finite-data error bounds have been developed for EDMD approximations in~\cite{Nuske2023}. Pointwise and uniform bounds are available for kernel-based approximations~\cite{Koehne2025}, as well as bounds proportional to the distance from the equilibrium point~\cite{Bold2025}. Robust stabilization approaches using linear matrix inequalities (LMI) and sum-of-squares have been developed for bilinear Koopman models of control-affine nonlinear systems~\cite{Strasser2025TAC,Strasser2025EJC}. Tube-based model predictive control (MPC) approaches have been used to handle Koopman approximation errors~\cite{Zhang2022}. Physics-based Koopman predictors have also been combined with data-driven residual models for MPC~\cite{bruder2025koopman}. Uniform bounds have been used to establish practical stability guarantees for linear parameter-varying (LPV)--Koopman MPC~\cite{LiuLPV2026}. Frequency-domain integral quadratic constraint(IQC) methods have been proposed to characterize residual errors in LPV--Koopman models~\cite{Eyuboglu2026}. 

For the output feedback,  the same approximation errors should be considered without having full state information. One approach for unknown systems in discrete time uses delay-embedded input--output states under uniform observability. They use bilinear Koopman models with proportional residual bounds to achieve stabilization~\cite{Strasser2026io}. For known dynamics, we seek a finite-dimensional LTI-plus-residual decomposition in continuous-time. Our goal is to obtain guarantees for the original nonlinear system using only the measured tracking error for output feedback.

Given nonlinear dynamics, including the actuator model, our approach evaluates the nominal directional derivatives of the chosen observables. Galerkin projection gives the state matrix and the closure residual. The actuator-realization defect and output-reconstruction error enter through distinct channels. Koopman eigenfunctions and polynomial or other continuously differentiable observables can be used under the regularity conditions defined below. The Galerkin projection follows the formulation in~\cite{Klus2020}, with inner products evaluated directly from the model as in~\cite{Asada2023Direct}.

The main contributions are as follows.
\begin{enumerate}

\item[(i)] We build nested Galerkin subspaces for the continuous-time actuator-augmented system that keep the actuator coordinates separate from the reduced observables. The construction requires no specific dictionary family and no invariant subspace of observables. The actuator block and the constant input matrix are preserved, and the finite-order nonclosure is kept as an exact additive residual for nonlinear dynamics $\dot{x}=f(x,u)$ including non-control-affine systems. Including the physical state coordinates in the lifted model gives an exact state readout. The bounds in (iv) then apply directly to the original state variables.

\item[(ii)] We then bound the closure residual, actuator-defect coupling, and the output-reconstruction error uniformly over the operating region. Exact linear constraints on the residual can be used to construct structured ellipsoidal enclosures when available. We express the Hilbert norm of the closure residual through best-approximation errors. We also derive a finite-order uniform bound.

\item[(iii)] We use the residual bounds to form an outer LTI system with bounded disturbance channels that contains every lifted nonlinear trajectory while the physical states stay in the operating region. Controller synthesis for a fixed realization is then a standard $\mathcal{H}_\infty$ output-feedback problem as in~\cite{DGKF1989}.

\item[(iv)] We obtain a quadratic storage function to bound the physical state and the tracking error explicitly before the first exit. If the physical-state image of the closed-loop ellipsoid, or an auxiliary bound on the physical state, lies inside the operating region, the trajectory remains there and is forward complete, with the physical state and tracking error uniformly ultimately bounded.

\end{enumerate}
\emph{Relation to prior work. }
Koopman direct encoding~\cite{Asada2023Direct} and control-coherent Koopman modeling~\cite{Asada2024,Asada2025} provide two ideas used in forming the nominal LTI part of our realization. At every retained order, we preserve the actuator block and the constant input matrix. What is added here is an explicit exact additive residual for finite-order drift nonclosure. The residual depends on the plant and actuator states, not on the control input of the actuator. Thus, neither a bilinear representation nor a control-affine plant is required. The LTI limitations in~\cite{Heeg2026} do not preclude this construction, because exactness is required of the LTI-plus-residual decomposition, not of its LTI part.

The bounds in~\cite{Nuske2023,Koehne2025,Bold2025} concern approximation errors in data-driven models. For a fixed dictionary, projection error can remain even in the infinite-data limit. We bound this finite-order closure error deterministically from the model, with attained maxima over the operating region.  Here, we work in the known-model setting. If such a model is not available, it can be identified from data~\cite{kurkccu2018mu} or built by combining physical modeling with data~\cite{schoukens2019nonlinear}. These extensions are outside the scope of this paper and would also require accounting for model error to establish guarantees for the true system.

For control-affine systems, \cite{Strasser2025TAC,Strasser2025EJC} combine bilinear models and error bounds with LMI or sum-of-squares state-feedback design. Our controller requires neither a control-affine plant nor full-state measurement. The controller design is conventional centralized $\mathcal{H}_\infty$ output feedback with frequency-domain weighting to shape tracking performance. Our residual bounds need not vanish at the equilibrium, and the closed-loop result is uniform ultimate boundedness rather than exponential stability.

For output feedback, the approach in~\cite{Strasser2026io} uses delay-embedded input--output states for discrete-time systems under a uniform nonlinear observability condition. In contrast, our method uses lifting only offline, and the feedback is a continuous-time linear dynamic controller driven solely by the current tracking error.

Sections~\ref{sec:prelim}--\ref{sec:lifting} develop the Galerkin
realization and residual analysis. Section~\ref{sec:control} establishes
the output-feedback guarantees. Section~\ref{sec:numerics} presents
regional certificates and comparisons with Jacobian-based
$\mathcal H_\infty$ feedback, backstepping, and Koopman-based hybrid MPC.
Section~\ref{sec:conclusion} concludes the article.

\section{Preliminaries}
\label{sec:prelim}
\label{sec:galerkin}
Let $X\subset X_{\mathrm{ext}}\subset\R^n$,
$U\subset U_{\mathrm{ext}}\subset\R^m$, and
$P\subset P_{\mathrm{ext}}\subset\R^\ell$, where
$X_{\mathrm{ext}}$, $U_{\mathrm{ext}}$, and $P_{\mathrm{ext}}$ are open.
Let
\[
f\in C(X_{\mathrm{ext}}\times U_{\mathrm{ext}},\R^n),
\
h\in C(X_{\mathrm{ext}},\R^{n_y}).
\]
We consider the class of nonlinear systems described by
\begin{equation}
\dot x=f(x,u),\ y=h(x),
\label{eq:plant}
\end{equation}
where $x\in X$ and $u\in U$ are the state and input, respectively. We also consider the following actuator-augmented representation as
\begin{equation}
\dot p=A_p p+B_pu_a+\delta_p,\ u=C_pp.
\label{eq:actuator}
\end{equation}
where $p\in P$ is the actuator state, $u_a\in\R^m$ is the actuator input,
and $\delta_p$ is the actuator-realization defect.
Here $A_p\in\R^{\ell\times\ell}$, $B_p\in\R^{\ell\times m}$,
and $C_p\in\R^{m\times\ell}$. We choose the physical
actuator-state set $P$ such that
$C_pP\subseteq U$. We assume
\begin{equation}
C_pB_p=0.
\label{eq:CpBp}
\end{equation}
\begin{remark}
\label{rem:CpBp-sufficient}
Condition~\eqref{eq:CpBp} is sufficient, but not necessary, for the input-coherence property used below. It can be satisfied, for instance, by a second-order low-pass filter, if needed, as a prefilter similar to~\cite{Asada2025}.
\end{remark}

Substituting $u$ from \eqref{eq:actuator} into
\eqref{eq:plant} gives
\begin{equation}
    \dot x=f(x,C_pp).
    \label{eq:plant-augmented}
\end{equation}

For an open set $O\subset\R^d$, $q\in\N$, and
$k\in\N\cup\{0\}$, let $C^k(O,\R^q)$ denotes the space of
$k$-times continuously differentiable maps from $O$ to $\R^q$.

If $\varphi\in C^1(O,\R)$ and $\vartheta\in O$, then
$D\varphi(\vartheta)\in\R^{1\times d}$ denotes the row-vector
representation of its Fr\'echet derivative at $\vartheta$ \cite[Sec.~4.2, Def.~1]{zeidler21995applied}. For product domains we use the corresponding partial derivatives
$D_x\varphi$, $D_u\varphi$, and $D_p\varphi$.

If $\mathbf F=(F_1,\dots,F_q)^\top\in C^1(O,\R^q)$, then
$D\mathbf F(\vartheta)\in\R^{q\times d}$ denotes the Jacobian matrix whose $i$-th row is $DF_i(\vartheta)$. For product domains, we use $D_x\mathbf F$, $D_u\mathbf F$, and $D_p\mathbf F$ for the corresponding partial Jacobians. If $\vartheta=(x,u)$, then $D_x\mathbf F(x,u)\in\R^{q\times n}$ and $D_u\mathbf F(x,u)\in\R^{q\times m}$. If $\vartheta=(p,x)$, then
$D_p\mathbf F(p,x)\in\R^{q\times\ell}$ and $D_x\mathbf F(p,x)\in\R^{q\times n}$.
\begin{definition}
\label{def:concrete-hilbert}
Let $\Omega\subset P_{\mathrm{ext}}\times X_{\mathrm{ext}}$ be compact, with $C_pp\in U_{\mathrm{ext}}$ for every $(p,x)\in\Omega$, and
let $\mu$ be a finite positive Borel measure with full support on
$\Omega$ (i.e., $\supp\mu=\Omega$). Set $\mathcal H=L^2(\Omega,\mu)$,
with inner product and norm
\[
\langle f,g\rangle_{\mathcal H}
\coloneqq\int_\Omega fg\,\mathrm d\mu,
\qquad
\|f\|_{\mathcal H}\coloneqq\langle f,f\rangle_{\mathcal H}^{1/2}.
\]
\end{definition}
\label{conv:continuous-representatives}
We identify continuous functions on $\Omega$ with their $L^2(\mu)$ classes. Since $\mu$ has full support, continuous representatives of the same class are equal pointwise on $\Omega$.
We use the following notation throughout. For a real Hilbert space
$\mathscr H$ and a closed subspace $V\subset\mathscr H$,
$\Pi_V:\mathscr H\to V$ denotes the orthogonal projection,
$\res_V\coloneqq I-\Pi_V$ the associated residual operator, and
$V^\perp$ the orthogonal complement. For $g\in\mathscr H$, set
\[
\dist_{\mathscr H}(g,V)
\coloneqq
\inf_{v\in V}\|g-v\|_{\mathscr H}
=
\|\res_Vg\|_{\mathscr H}.
\] 
The product space
$\mathscr H^m$ is equipped with the canonical inner product
\[
\langle \mathbf f,\mathbf g\rangle_{\mathscr H^m}
=
\sum_{i=1}^{m}\langle f_i,g_i\rangle_{\mathscr H},
\]
and the induced norm
$\|\mathbf f\|_{\mathscr H^m}
=(\sum_{i=1}^m\|f_i\|_{\mathscr H}^2)^{1/2}$.
If $E,F\subset\mathscr H$ satisfy $E\cap F=\{0\}$, then
$E\oplus F$ denotes the algebraic direct sum $E+F$.
Equivalently, every element of $E\oplus F$ has a unique
representation $e+f$, with $e\in E$ and $f\in F$
\cite[I, Section~4]{lang1987linear}.
If $\mathbf F=(F_1,\ldots,F_q)^\top\in C(\Omega,\R^q)$, then
$\mathbf F$ denotes the function-valued vector and
$\mathbf F(\vartheta)\in\R^q$ its value at $\vartheta\in\Omega$. 

For $g\in C(\Omega,\R)$, define $\|g\|_{L^\infty(\Omega)}
\coloneqq \max_{\vartheta\in\Omega}|g(\vartheta)|$.

Finite-dimensional vectors are equipped with the Euclidean norm $\|\cdot\|_2$, and real matrices with the induced norm
\mbox{$\|\cdot\|_{2\to2}$}. If
$T:C(\Omega,\R)\to C(\Omega,\R)$ is bounded and linear, then
\mbox{$\|T\|_{C(\Omega)\to C(\Omega)}$} denotes the operator norm induced by
\mbox{$\|\cdot\|_{L^\infty(\Omega)}$}.

Functions depending only on a subset of the coordinates of $\Omega$
are identified with their natural extensions to $\Omega$. In particular, for
$g=g(x)$ defined on $X_{\mathrm{ext}}$,
$(\iota_xg)(p,x):=g(x)$ for $(p,x)\in\Omega$.
The same convention applies component-wise on vector-valued functions and the embedding notation is omitted whenever no ambiguity arises.
The following Galerkin projection formulas define the finite-dimensional matrices used below. If $V\subset\mathscr H$ is a
closed subspace, then $V^m\subset\mathscr H^m$ is closed and
\[
\begin{aligned}
\Pi_{V^m}\mathbf g&=(\Pi_Vg_1,\ldots,\Pi_Vg_m)^\top,\\
\res_{V^m}\mathbf g&=(\res_Vg_1,\ldots,\res_Vg_m)^\top.
\end{aligned}
\]

\begin{definition}
\label{def:QR}
Let $\mathscr H$ be a real Hilbert space. Given
$\mathbf F=(f_1,\dots,f_q)^\top\in\mathscr H^q$ and
$\mathbf G=(g_1,\dots,g_m)^\top\in\mathscr H^m$, set
$V\coloneqq\Span\{f_1,\dots,f_q\}\subset\mathscr H$.
Define the Gram matrix and cross-Gram matrix by
\begin{equation}
\begin{aligned}
R(\mathbf F)
\coloneqq
\bigl(\langle f_i,f_j\rangle_{\mathscr H}\bigr)_{i,j=1}^q
\in\R^{q\times q},
\label{eq:R-general}
\end{aligned}
\end{equation}
and
\begin{equation}
\begin{aligned}
Q(\mathbf G,\mathbf F)
\coloneqq
\bigl(\langle g_i,f_j\rangle_{\mathscr H}\bigr)_{\substack{i=1,\dots,m\\ j=1,\dots,q}}
\in\R^{m\times q}.
\label{eq:Q-general}
\end{aligned}
\end{equation}
\end{definition}

\begin{lemma}
\label{lem:gram-invertible}
\label{thm:general-proj}
Let $\mathscr H$ be a real Hilbert space and
$\mathbf F=(f_1,\dots,f_q)^\top\in\mathscr H^q$.
Then $R(\mathbf F)\succ0$ if and only if $f_1,\dots,f_q$ are linearly
independent in $\mathscr H$. In this case, $V=\Span\{f_1,\dots,f_q\}$
is closed and, for every $\mathbf G\in\mathscr H^m$,
\begin{equation}
\Pi_{V^m}\mathbf G
=
Q(\mathbf G,\mathbf F)R(\mathbf F)^{-1}\mathbf F.
\label{eq:general-proj}
\end{equation}
\end{lemma}
\begin{proof}
For $a\in\R^q$, $a^\top R(\mathbf F)a=\|\sum_i a_if_i\|_{\mathscr H}^2$,
which proves the first claim \cite[Theorem~6.5-1]{kreyszig1978introductory}.
Finite-dimensional subspaces are closed. The projection formula follows
componentwise from the Gram normal equations
\cite[Ch.~3, p.~80]{arbogastbona}.
\end{proof}
For
$\langle f,g\rangle_{\mathcal H}
=\int_\Omega fg\,\mathrm d\mu$,
the Gram and cross-Gram matrices take the form
\[
R(\mathbf F)
=\int_\Omega \mathbf F\mathbf F^\top\,\mathrm d\mu,\quad
Q(\mathbf G,\mathbf F)
=\int_\Omega \mathbf G\mathbf F^\top\,\mathrm d\mu.
\]

\section{Control-coherent finite-dimensional Galerkin lifting with additive uncertainty}
\label{sec:lifting}
In this section we develop a Galerkin construction whose main result is an exact finite-dimensional LTI-plus-residual decomposition that preserves the actuator block and constant input matrix at every retained order. We first introduce the actuator-composed lift family and the regularity condition used, together with \eqref{eq:CpBp}, to establish this result.

Let
$\{\psi_k\}_{k\ge1}\subset
C^1(X_{\mathrm{ext}}\times U_{\mathrm{ext}},\R)$
be a prescribed family of scalar lift functions. For each $k\ge1$, define the raw actuator-composed lift on $\Omega$ by
\begin{equation}
\phi_k^{\mathrm{raw}}(p,x)
\coloneqq
\psi_k(x,C_pp),
\label{eq:raw-lift}
\end{equation}
and its nominal Lie derivative
\begin{equation}
\begin{aligned}
\gamma_k^{\mathrm{raw}}(p,x)
&\coloneqq D_x\psi_k(x,C_pp)f(x,C_pp)\\
&\quad+D_u\psi_k(x,C_pp)C_pA_pp.
\label{eq:raw-gamma}
\end{aligned}
\end{equation}
Define also the actuator coordinate functions
\begin{equation}
\pi_i(p,x)\coloneqq p_i, \ \textit{for }
i=1,\dots,\ell.
\label{eq:pi-scalar}
\end{equation}
Since $\psi_k\in C^1$ and $f$ is continuous, the raw lifts and their nominal Lie derivatives are continuous on $\Omega$. To use the actuator coordinates as basis functions, we make the following assumption.
\begin{assumption}
\label{ass:CCK-membership}
The functions $\pi_1,\dots,\pi_\ell$ are linearly independent in
$\mathcal H$.
\end{assumption}
Since $\mu$ has full support, the previous assumption is equivalent to requiring that $a^\top p=0$ on $\Omega$ imply $a=0$. In other words, no nonzero linear combination of the actuator coordinates vanishes on $\Omega$. This holds whenever $\Omega$ has nonempty interior. Accordingly, define the actuator-coordinate subspace
\[
U_p
\coloneqq
\Span\{\pi_1,\dots,\pi_\ell\}
\subset\mathcal H.
\]

To fix the representatives used in the derivative calculations below,
define
\[
\Xi(p,x)\coloneqq(x,C_pp),
\\
\mathcal O
\coloneqq
\Xi^{-1}(X_{\mathrm{ext}}\times U_{\mathrm{ext}})
\subset P_{\mathrm{ext}}\times X_{\mathrm{ext}}.
\]
Since $\Xi$ is continuous, $\mathcal O$ is open. Moreover, the choice of $\Omega$ in Definition~\ref{def:concrete-hilbert} gives $\Omega\subset\mathcal O$. Thus $\mathcal O$ is an open neighborhood
of $\Omega$. Since $\psi_k\in C^1$ and $\Xi$ is smooth, each raw lift
has the canonical $C^1$ extension
\[
\phi_{k,\mathrm{ext}}^{\mathrm{raw}}(p,x)
\coloneqq
\psi_k(x,C_pp),
\\
(p,x)\in\mathcal O.
\]
For each reduced lift constructed below, we retain its defining finite
expansion
\[
\widehat\phi_k
=
\sum_{r=1}^{m_k}c_{kr}\phi_{n_{kr}}^{\mathrm{raw}}
\]
and use the same coefficients to define its extension on $\mathcal O$.
All derivatives of the raw and reduced lifts below are taken with
respect to these extensions.

\begin{theorem}
\label{thm:reduced-family}
Under Assumption~\ref{ass:CCK-membership}, the raw lift family can be
reduced to a nested family that is linearly independent of the
actuator-coordinate subspace. More precisely, there exist
$M\in\N\cup\{0,\infty\}$, a strictly increasing sequence of indices
$\{j_k\}_{k=1}^{M}\subset\N$ (empty if $M=0$), and an
$\mathcal H$-orthonormal family
$\{\widehat\phi_k\}_{k=1}^{M}\subset C(\Omega,\R)$ such that, for every
$N$ with $1\le N\le M$,
\begin{equation}
\begin{aligned}
V_N
\coloneqq
\Span\{\widehat\phi_1,\dots,\widehat\phi_N\}
=
\Span\{\phi_{j_1}^{\mathrm{raw}},\dots,\phi_{j_N}^{\mathrm{raw}}\}
\subset\mathcal H,
\label{eq:VCCK-WN}
\end{aligned}
\end{equation}
and $U_p\cap V_N=\{0\}$.
Consequently,
\begin{equation}
\begin{aligned}
W_N
\coloneqq
U_p\oplus V_N
=
U_p+\Span\{\phi_{j_1}^{\mathrm{raw}},\dots,\phi_{j_N}^{\mathrm{raw}}\},
\label{eq:WN-characterization}
\end{aligned}
\end{equation}
and $W_N\subset W_{N+1}$ whenever $1\le N<M$. In particular, each
$\widehat\phi_k$ is a finite real linear combination of raw lifts.

If $M=0$, then all raw lifts belong to $U_p$. If
$1\le M<\infty$, then
\begin{equation}
\begin{aligned}
U_p+\Span\{\phi_k^{\mathrm{raw}}:k\ge1\}
=
U_p\oplus V_M.
\label{eq:finite-M}
\end{aligned}
\end{equation}
Hence, if $\bigl(U_p+\Span\{\phi_k^{\mathrm{raw}}:k\ge1\}\bigr)/U_p$
is infinite-dimensional, then $M=\infty$.
\end{theorem}

\begin{proof}
If every raw lift belongs to $U_p$, set $M=0$ and the conclusion is
immediate. Otherwise, let $j_1$ be the smallest index such that
$\phi_{j_1}^{\mathrm{raw}}\notin U_p$ and set
\[
\widehat\phi_1
\coloneqq
\frac{\phi_{j_1}^{\mathrm{raw}}}
{\|\phi_{j_1}^{\mathrm{raw}}\|_{\mathcal H}}.
\]
Then $V_1=\Span\{\phi_{j_1}^{\mathrm{raw}}\}$ and
$U_p\cap V_1=\{0\}$.

Suppose inductively that $j_1<\cdots<j_N$ and
$\widehat\phi_1,\dots,\widehat\phi_N$ have been constructed so that
they are orthonormal,
\[
V_N
=
\Span\{\phi_{j_1}^{\mathrm{raw}},\dots,\phi_{j_N}^{\mathrm{raw}}\},
\\
U_p\cap V_N=\{0\}.
\]
If every $j>j_N$ satisfies $\phi_j^{\mathrm{raw}}\in U_p\oplus V_N$, the recursion terminates and
we set $M=N$. Otherwise, let $j_{N+1}>j_N$ be the smallest index for
which $\phi_{j_{N+1}}^{\mathrm{raw}}\notin U_p\oplus V_N$,
and define
\[
\widetilde\phi_{N+1}
\coloneqq
\phi_{j_{N+1}}^{\mathrm{raw}}
-
\Pi_{V_N}\phi_{j_{N+1}}^{\mathrm{raw}},
\\
\widehat\phi_{N+1}
\coloneqq
\frac{\widetilde\phi_{N+1}}
{\|\widetilde\phi_{N+1}\|_{\mathcal H}}.
\]
Since $\phi_{j_{N+1}}^{\mathrm{raw}}\notin V_N$, we have $\widetilde\phi_{N+1}\neq0$. Also, $\widetilde\phi_{N+1}\in V_N^\perp$, so the enlarged family is orthonormal. The new lift is continuous, being a finite combination of raw lifts. Moreover, the induction hypothesis
gives
\[
\Pi_{V_N}\phi_{j_{N+1}}^{\mathrm{raw}}
\in
\Span\{\phi_{j_1}^{\mathrm{raw}},\dots,\phi_{j_N}^{\mathrm{raw}}\},
\]
and therefore
\[
V_{N+1}
=
\Span\{\phi_{j_1}^{\mathrm{raw}},\dots,
       \phi_{j_{N+1}}^{\mathrm{raw}}\}.
\]

It remains to verify directness. If $u+v+a\widehat\phi_{N+1}=0,
u\in U_p,\ v\in V_N$, and $a\neq0$, then
$\widehat\phi_{N+1}\in U_p\oplus V_N$. Since
\[
\phi_{j_{N+1}}^{\mathrm{raw}}
=
\|\widetilde\phi_{N+1}\|_{\mathcal H}\widehat\phi_{N+1}
+
\Pi_{V_N}\phi_{j_{N+1}}^{\mathrm{raw}},
\]
this would imply
$\phi_{j_{N+1}}^{\mathrm{raw}}\in U_p\oplus V_N$, contradicting the
choice of $j_{N+1}$. Hence $a=0$, and then
$u=-v\in U_p\cap V_N=\{0\}$. Thus
$U_p\cap V_{N+1}=\{0\}$. The nesting
$V_N\subset V_{N+1}$ immediately gives $W_N\subset W_{N+1}$.

If the construction never terminates, set $M=\infty$. If it stops
after $M<\infty$ steps, every selected raw lift belongs to $V_M$,
every skipped index belongs to the space available when it was skipped
and hence to $U_p\oplus V_M$, and every subsequent raw lift belongs to
$U_p\oplus V_M$ by the stopping criterion. This proves
\eqref{eq:finite-M}. A finite $M$ makes the corresponding quotient
finite-dimensional, so an infinite-dimensional quotient forces
$M=\infty$.
\end{proof}
This reduction gives a basis for $W_N$ that includes the actuator coordinates.
The recursive construction also implies that, for every
$1\le N\le M$,
\begin{equation}
\begin{aligned}
\phi_j^{\mathrm{raw}}\in W_N \quad \forall j \leq j_N
\label{eq:raw-prefix-in-WN}
\end{aligned}
\end{equation}
Indeed, each selected lift $\phi_{j_r}^{\mathrm{raw}}$ belongs to
$V_N\subset W_N$, while every unselected index between $j_r$ and $j_{r+1}$ was skipped because its raw lift already belonged to
$W_r\subset W_N$; indices below $j_1$ belong to $U_p\subset W_N$ by
the minimality of $j_1$.
\begin{remark}
\label{rem:reduced-family-order}

The finite-order spaces $V_N,W_N$ and their matrices and residuals
depend on the ordering of the raw lifts. This ordering may be chosen
to retain physically useful observables at low orders.

\end{remark}
\begin{definition}
\label{def:CCK-reduced}
Let $M\in\N\cup\{0,\infty\}$ and
$\{\widehat\phi_k\}_{k=1}^{M}$ be as in
Theorem~\ref{thm:reduced-family}. For $1\le N\le M$, define
\begin{equation}
\begin{aligned}
\mathbf\Phi_N
\coloneqq
\col(\widehat\phi_1,\dots,\widehat\phi_N)
\in C(\Omega,\R^N),
\label{eq:PhiN}
\end{aligned}
\end{equation}
\begin{equation}
\begin{aligned}
\bm\pi
\coloneqq
\col(\pi_1,\dots,\pi_\ell)
\in C(\Omega,\R^\ell),
\quad
\bm\pi(p,x)=p,
\label{eq:pi-vector}
\end{aligned}
\end{equation}
and
\begin{equation}
\begin{aligned}
\mathbf Z_N
\coloneqq
\col(\bm\pi,\mathbf\Phi_N)
\in C(\Omega,\R^{\ell+N}).
\label{eq:ZN}
\end{aligned}
\end{equation}

For $1\le k\le M$, define the nominal Lie derivative
\begin{equation}
\begin{aligned}
\widehat\gamma_k(p,x)
\coloneqq
D_p\widehat\phi_k(p,x)A_pp
+
D_x\widehat\phi_k(p,x)f(x,C_pp).
\label{eq:gammahat}
\end{aligned}
\end{equation}
For $1\le N\le M$, set
\begin{equation}
\begin{aligned}
\boldsymbol\gamma_N(p,x)
&\coloneqq
\col\bigl(
\widehat\gamma_1(p,x),\dots,\widehat\gamma_N(p,x)
\bigr)\\
&=
\underbrace{D_p\mathbf\Phi_N(p,x)}_{\textstyle =:\widehat J_N(p,x)\in\R^{N\times\ell}}A_pp\\
&\quad+D_x\mathbf\Phi_N(p,x)f(x,C_pp).
\end{aligned}
\label{eq:GammaCCK}
\end{equation}
Define also the augmented nominal derivative field
\begin{equation}
\begin{aligned}
\mathbf\Gamma_N(p,x)
\coloneqq
\col\bigl(
A_p\bm\pi(p,x),
\boldsymbol\gamma_N(p,x)
\bigr)
\in\R^{\ell+N}.
\label{eq:Upsilon}
\end{aligned}
\end{equation}
\end{definition}

\begin{lemma}
\label{lem:lower-block-exactB}
For each $k$ with $1\le k\le M$,
$\widehat\gamma_k\in C(\Omega,\R)$ and
\begin{equation}
\begin{aligned}
D_p\widehat\phi_k(p,x)B_p=0
\ \text{for all }(p,x)\in\Omega.
\label{eq:exactB-scalar}
\end{aligned}
\end{equation}
Consequently, for every $1\le N\le M$, $D_p\mathbf\Phi_N(p,x)B_p=0$
on $\Omega$, and $\boldsymbol\gamma_N$ and $\mathbf\Gamma_N$ are continuous.
\end{lemma}

\begin{proof}
Fix $k$ and use its defining finite expansion
\[
\widehat\phi_k
=
\sum_{r=1}^{m_k}c_{kr}\phi_{n_{kr}}^{\mathrm{raw}}.
\]
For each raw lift,
\[
\begin{aligned}
D_p\phi_{n_{kr}}^{\mathrm{raw}}(p,x)&=D_u\psi_{n_{kr}}(x,C_pp)C_p,\\
D_x\phi_{n_{kr}}^{\mathrm{raw}}(p,x)&=D_x\psi_{n_{kr}}(x,C_pp).
\end{aligned}
\]
Hence
\[
D_p\phi_{n_{kr}}^{\mathrm{raw}}A_pp
+
D_x\phi_{n_{kr}}^{\mathrm{raw}}f(x,C_pp)
=
\gamma_{n_{kr}}^{\mathrm{raw}},
\]
and therefore, by linearity,
\[
\widehat\gamma_k
=
\sum_{r=1}^{m_k}c_{kr}\gamma_{n_{kr}}^{\mathrm{raw}}
\]
is continuous on $\Omega$ by the stated regularity assumptions. Moreover,
\[
D_p\widehat\phi_k(p,x)B_p
=
\sum_{r=1}^{m_k}
c_{kr}D_u\psi_{n_{kr}}(x,C_pp)C_pB_p
=
0
\]
by \eqref{eq:CpBp}, proving \eqref{eq:exactB-scalar}.
Stacking gives the claimed Jacobian identity and continuity of
$\boldsymbol\gamma_N$ and $\mathbf\Gamma_N$.
\end{proof}
Fix $N$ with $1\le N\le M$. Define the Gram and cross-Gram matrices
\begin{equation}
\begin{aligned}
R_N
\coloneqq
R(\mathbf Z_N)
\in\R^{(\ell+N)\times(\ell+N)},
\label{eq:RCCK}
\end{aligned}
\end{equation}
\begin{equation}
\begin{aligned}
Q_N
\coloneqq
Q(\mathbf\Gamma_N,\mathbf Z_N)
\in\R^{(\ell+N)\times(\ell+N)}.
\label{eq:QCCK}
\end{aligned}
\end{equation}
By~\eqref{eq:actuator} and Lemma~\ref{lem:lower-block-exactB}, the lifted input matrix is
\begin{equation}
\begin{aligned}
B_N
\coloneqq
\begin{bmatrix}
B_p\\
0_{N\times m}
\end{bmatrix}
\in\R^{(\ell+N)\times m}.
\label{eq:BCCK}
\end{aligned}
\end{equation}

By Assumption~\ref{ass:CCK-membership} and
Theorem~\ref{thm:reduced-family}, the components of $\mathbf Z_N$
form a basis of $W_N=U_p\oplus V_N$. Thus
$R_N=R(\mathbf Z_N)$ is invertible by
Lemma~\ref{lem:gram-invertible}.
Accordingly, define the Galerkin matrix
\begin{equation}
\begin{aligned}
A_N
\coloneqq
Q_NR_N^{-1}
\in\R^{(\ell+N)\times(\ell+N)}.
\label{eq:GCCK}
\end{aligned}
\end{equation}
Since $\mathbf\Gamma_N$ is continuous by
Lemma~\ref{lem:lower-block-exactB} and the components of
$\mathbf Z_N$ form a basis of $W_N$,
Lemma~\ref{thm:general-proj} gives
\[
\Pi_{W_N^{\ell+N}}\mathbf\Gamma_N
=
Q_NR_N^{-1}\mathbf Z_N
=
A_N\mathbf Z_N.
\]
For a locally well-posed nominal flow, the Lie derivatives in
$\mathbf\Gamma_N$ give the action of the Koopman
generator~\cite{Klus2020}. Thus, $A_N$ is the corresponding
Galerkin matrix.

The first $\ell$ components of $\mathbf\Gamma_N$ are those of
$A_p\bm\pi$ and belong to $U_p\subset W_N$, so componentwise
projection leaves them unchanged. Thus
\[
(A_N\mathbf Z_N)_{1:\ell}=A_p\bm\pi.
\]
Writing
\[
A_N=
\begin{bmatrix}
A_{11,N} & A_{12,N}\\
A_{21,N} & A_{22,N}
\end{bmatrix}
\]
conformably with $\mathbf Z_N=\col(\bm\pi,\mathbf\Phi_N)$ gives
\[
(A_{11,N}-A_p)\bm\pi+A_{12,N}\mathbf\Phi_N=0.
\]
Since the components of $\mathbf Z_N$ are linearly independent,
$A_{11,N}=A_p$ and $A_{12,N}=0$. Therefore
\begin{equation}
A_N=
\begin{bmatrix}
A_p & 0_{\ell\times N}\\
A_{21,N} & A_{22,N}
\end{bmatrix},
\label{eq:GCCK-block}
\end{equation}
where $A_{21,N}\in\R^{N\times\ell}$ and $A_{22,N}\in\R^{N\times N}$.

\begin{theorem}
\label{thm:main-CCK}
Fix a reduced order $N$ with $1\le N\le M$. Define the closure
residual field by
\begin{equation}
\begin{aligned}
\mathbf r_N
\coloneqq
\res_{W_N^N}\boldsymbol\gamma_N
\in C(\Omega,\R^N).
\label{eq:rCCK}
\end{aligned}
\end{equation}
\begin{enumerate}[label=\textnormal{(\roman*)}]
\item
The following field identity holds pointwise on $\Omega$:
\begin{equation}
\mathbf\Gamma_N
=
A_N\mathbf Z_N
+
\col\bigl(0_\ell,\mathbf r_N\bigr).
\label{eq:Upsilon-split}
\end{equation}
    \item if $I\subset\R$ is an interval and
    $(p(\cdot),x(\cdot),u_a(\cdot))
    \in C^1(I;P\times X\times\R^m)$ satisfies
\begin{equation}
\begin{aligned}
\dot p(t)
&= A_pp(t)+B_pu_a(t)+\delta_p(t),\\
\dot x(t)
&= f\bigl(x(t),C_pp(t)\bigr), \ (p(t),x(t))
\in\Omega
\end{aligned}
\label{eq:trajectory-CCK}
\end{equation}
    for all $t\in I$, then the evaluated lifted state
\begin{equation}
\begin{aligned}
    \zeta_N(t)
    \coloneqq
    \mathbf Z_N(p(t),x(t))
    \in\R^{\ell+N}
    \label{eq:zeta}
\end{aligned}
\end{equation}
   satisfies
\begin{equation}
\dot\zeta_N(t)
=
A_N\zeta_N(t)
+
B_Nu_a(t)
+
\Delta_N(t)
\label{eq:main-CCK-ODE}
\end{equation}
for all $t\in I$, where
\begin{equation}
\begin{aligned}
    \Delta_N(t)
    \coloneqq
    \begin{bmatrix}
    \delta_p(t)\\
    \mathbf r_N(p(t),x(t))
    +
    \widehat J_N(p(t),x(t))\delta_p(t)
    \end{bmatrix}.
    \label{eq:Delta}
\end{aligned}
\end{equation}
\end{enumerate}
\end{theorem}

\begin{proof}
The componentwise projection computed above leaves $A_p\bm\pi$
unchanged and has lower-block residual $\mathbf r_N$ by
\eqref{eq:rCCK}. This gives \eqref{eq:Upsilon-split}.
The standing convention makes this identity pointwise on $\Omega$.

Now let $(p(\cdot),x(\cdot),u_a(\cdot))$ satisfy
\eqref{eq:trajectory-CCK}. The upper block is immediate:
\[
\frac{\mathrm d}{\mathrm dt}\bm\pi(p(t),x(t))
=
\dot p(t)
=
A_pp(t)+B_pu_a(t)+\delta_p(t).
\]
For the lower block, fix $k\in\{1,\dots,N\}$. Using the canonical
extensions on $\mathcal O$, the chain rule gives
\[
\begin{aligned}
\frac{\mathrm d}{\mathrm dt}
\widehat\phi_k(p(t),x(t))
&=
D_p\widehat\phi_k(p(t),x(t))\dot p(t)\\
&\quad +D_x\widehat\phi_k(p(t),x(t))\dot x(t).
\end{aligned}
\]
Substitution of~\eqref{eq:trajectory-CCK} and
Lemma~\ref{lem:lower-block-exactB} eliminates the $B_pu_a$ term and gives
\[
\frac{d}{dt}\widehat\phi_k(p,x)
=\widehat\gamma_k(p,x)+D_p\widehat\phi_k(p,x)\delta_p
\]
along the trajectory.
Stacking over $k=1,\dots,N$ yields
\begin{equation}
\begin{aligned}
\frac{\mathrm d}{\mathrm dt}
\mathbf\Phi_N(p(t),x(t))
=
\boldsymbol\gamma_N(p(t),x(t))
+
\widehat J_N(p(t),x(t))\delta_p(t).
\label{eq:lower-block}
\end{aligned}
\end{equation}
Evaluating \eqref{eq:Upsilon-split} along the trajectory identifies
the nominal projected part as $A_N\zeta_N(t)$. Combining the upper
block, the exact input term $B_Nu_a(t)$, and the uncertainty terms
from \eqref{eq:lower-block} proves
\eqref{eq:main-CCK-ODE}--\eqref{eq:Delta}.
\end{proof}

\begin{definition}
\label{cor:uniform-uncertainty}
Fix $1\le N\le M$. By continuity and compactness, the maxima
below are finite and attained. These are the smallest scalar bounds for the fixed dictionary,
coordinates, region, and measure. Define
\begin{equation}
\begin{aligned}
\rho_N^{\star}
&\coloneqq\max_{\vartheta\in\Omega}\|\mathbf r_N(\vartheta)\|_2,\\
L_N^{\star}
&\coloneqq\max_{\vartheta\in\Omega}\|\widehat J_N(\vartheta)\|_{2\to2},
\label{eq:rho-L}
\end{aligned}
\end{equation}
and, for $\delta\ge0$,
\begin{equation}
\sigma_{c,N}^{\star}(\delta)
\coloneqq
\max_{\substack{\vartheta\in\Omega\\d\in\R^\ell,\ \|d\|_2\le\delta}}
\|\mathbf r_N(\vartheta)+\widehat J_N(\vartheta)d\|_2.
\label{eq:sigma-c-s6}
\end{equation}
The triangle inequality gives the following bound:
\begin{equation}
\sigma_{c,N}^{\star}(0)=\rho_N^\star,
\qquad
\sigma_{c,N}^{\star}(\delta)\le\rho_N^\star+L_N^\star\delta.
\label{eq:sigma-c-properties}
\end{equation}
\end{definition}
\begin{remark}
\label{rem:weighted-region-dependence}
The Galerkin closure and residual are determined by the chosen
$(\Omega,\mu)$ and are therefore regional and measure weighted,
whereas the exact actuator blocks $A_p$ and $B_N$ are independent of
this weighting.
\end{remark}

\begin{corollary}
\label{cor:cck-readouts}
Fix $N$ with $1\le N\le M$.
\begin{enumerate}[label=\textnormal{(\roman*)}]

\item
Using the standing embedding convention, define
\begin{equation}
\begin{aligned}
C_{y,N}
\coloneqq
\left(
\int_\Omega h Z_N^\top\,d\mu
\right)R_N^{-1}
\mathrel{\in}\R^{n_y\times(\ell+N)}.
\label{eq:cck-output-readout}
\end{aligned}
\end{equation}
Set $v_{y,N}
\coloneqq
h-C_{y,N}\mathbf Z_N
\in C(\Omega,\R^{n_y})$.
Then
\begin{equation}
\begin{aligned}
h
=
C_{y,N}\mathbf Z_N+v_{y,N}
\label{eq:cck-output-readout-residual}
\end{aligned}
\end{equation}
pointwise on $\Omega$, and
\begin{equation}
\begin{aligned}
\nu_N^{\star}
\coloneqq
\max_{\vartheta\in\Omega}
\|v_{y,N}(\vartheta)\|_2
<\infty.
\label{eq:nu-star}
\end{aligned}
\end{equation}

\item
Suppose that, for every $i\in\{1,\dots,n\}$, there exist
$k_i\le j_N$ and $a_i\ne0$ such that
\begin{equation}
\begin{aligned}
\psi_{k_i}(x,u)=a_i x_i.
\label{eq:state-coordinate-retention}
\end{aligned}
\end{equation}
Then $x\in W_N^n$, and
\begin{equation}
\begin{aligned}
H_{x,N}
\coloneqq
Q(x,\mathbf Z_N)R_N^{-1}
\in\R^{n\times(\ell+N)}
\label{eq:HxN}
\end{aligned}
\end{equation}
satisfies
\begin{equation}
x = H_{x,N}\mathbf Z_N(p,x),
\quad \forall (p,x)\in\Omega.
\label{eq:exact-state-readout}
\end{equation}
Consequently, defining $\xi\coloneqq\col(p,x)$ and
\begin{equation}
\begin{aligned}
H_{\xi,N}
\coloneqq
\begin{bmatrix}
\begin{array}{cc}
I_\ell & 0_{\ell\times N}
\end{array}\\[1mm]
H_{x,N}
\end{bmatrix}
\in\R^{(\ell+n)\times(\ell+N)},
\label{eq:HxiN}
\end{aligned}
\end{equation}
one has
\begin{equation}
\begin{aligned}
\xi
=
H_{\xi,N}\mathbf Z_N(p,x),\ \forall(p,x)\in\Omega.
\label{eq:exact-augmented-state-readout}
\end{aligned}
\end{equation}
\end{enumerate}
\end{corollary}
\begin{proof}
For \textnormal{(i)}, the integral in
\eqref{eq:cck-output-readout} is well defined because $h$ and
$\mathbf Z_N$ are continuous on the compact set $\Omega$ and $\mu$ is
finite, and $R_N$ is invertible as established above.
The definition of $v_{y,N}$ gives
\eqref{eq:cck-output-readout-residual} and
$v_{y,N}\in C(\Omega,\R^{n_y})$. 
For \textnormal{(ii)}, \eqref{eq:state-coordinate-retention} gives
\[
\phi_{k_i}^{\mathrm{raw}}(p,x)
=
\psi_{k_i}(x,C_pp)
=
a_i x_i.
\]
Since $k_i\le j_N$, \eqref{eq:raw-prefix-in-WN} implies
$\phi_{k_i}^{\mathrm{raw}}\in W_N$, and hence $x_i\in W_N$.
Thus $x\in W_N^n$. Since the components of $\mathbf Z_N$
form a basis of $W_N$, Lemma~\ref{thm:general-proj} gives
\[
x
=
\Pi_{W_N^n}x
=
Q(x,\mathbf Z_N)R_N^{-1}\mathbf Z_N
=
H_{x,N}\mathbf Z_N
\]
in $\mathcal H^n$. The standing convention gives
\eqref{eq:exact-state-readout} pointwise on $\Omega$.
Finally, \eqref{eq:exact-augmented-state-readout} follows from
$\bm\pi(p,x)=p$.
\end{proof}
We next express the finite-order closure residual through best
approximation over $W_N$ and give a uniform bound.

\begin{lemma}
\label{thm:closure-criterion}
For every $1\le N\le M$,
\begin{equation}
\begin{aligned}
\|\mathbf r_N\|_{\mathcal H^N}^2
=
\sum_{k=1}^N
\dist_{\mathcal H}^2(\widehat\gamma_k,W_N).
\label{eq:best-approx}
\end{aligned}
\end{equation}
\end{lemma}
\begin{proof}
By \eqref{eq:rCCK}, each residual component has norm
$\dist_{\mathcal H}(\widehat\gamma_k,W_N)$. Summing their squared norms
proves \eqref{eq:best-approx}.
\end{proof}
When $M=\infty$, consistency of each fixed mode means
\begin{equation}
\dist_{\mathcal H}(\widehat\gamma_k,W_N)\to0
\quad(N\to\infty),\qquad k\in\N\ \text{fixed}.
\label{eq:pointwise-distance}
\end{equation}
Since $W_N$ are nested, this is equivalent to
$\widehat\gamma_k\in\overline{\bigcup_{N\ge1}W_N}$ for every fixed $k$,
where the overline denotes closure in the $\mathcal H$ norm.
Appendix~\ref{app:growing-residual} gives an additional summability condition for $\|\mathbf r_N\|_{\mathcal H^N}\to0$.
\begin{proposition}
\label{cor:uniform-closure}
Fix $N$ with $1\le N\le M$. For $g\in C(\Omega,\R)$, define
\begin{equation}
\widetilde\Pi_{W_N}g
\coloneqq
\left(
\int_\Omega g Z_N^\top\,d\mu
\right)R_N^{-1}Z_N.
\label{eq:Pi-tilde}
\end{equation}
Then $\widetilde\Pi_{W_N}:C(\Omega,\R)\to C(\Omega,\R)$ is a bounded linear projection with range $W_N$. Moreover, for every
$g\in C(\Omega,\R)$,
\begin{equation}
\widetilde\Pi_{W_N}g
=
\Pi_{W_N}g
\quad\text{in }\mathcal H.
\label{eq:Pi-coincides}
\end{equation}
Define
\begin{equation}
\Lambda_N
\coloneqq
\|\widetilde\Pi_{W_N}\|_{C(\Omega)\to C(\Omega)}.
\label{eq:LambdaN}
\end{equation}
Then
\begin{equation}
\rho_N^\star
\le
(1+\Lambda_N)
\left(
\sum_{k=1}^{N}
\inf_{w\in W_N}
\|\widehat\gamma_k-w\|_{L^\infty(\Omega)}^2
\right)^{1/2}.
\label{eq:uniform-closure-bound}
\end{equation}
\end{proposition}

The proof is given in Appendix~\ref{app:uniform-closure-proof}.

Uniform decay follows if the right-hand side of
\eqref{eq:uniform-closure-bound} tends to zero. The next result gives a structured ellipsoidal enclosure of the finite-order residual.
\begin{proposition}
\label{prop:shaped-residual}
Fix $1\le N\le M$ and the residual $\mathbf r_N$ in
\eqref{eq:rCCK}. Let $F_N$ be a real matrix with $N$ columns such that $F_N\mathbf r_N=0
\quad\text{on }\Omega$. Let $q\coloneqq\dim\ker F_N>0$, and let
$U_N\in\R^{N\times q}$ have orthonormal columns spanning $\ker F_N$. Suppose $S_N=S_N^\top\succ0$ in $\R^{q\times q}$ satisfies
\[
\bigl(U_N^\top\mathbf r_N(\vartheta)\bigr)^\top
S_N^{-1}
\bigl(U_N^\top\mathbf r_N(\vartheta)\bigr)
\le1,
\qquad
\vartheta\in\Omega.
\]
Then, with $D_N\coloneqq U_NS_N^{1/2},
\
\eta_N(\vartheta)
\coloneqq
S_N^{-1/2}U_N^\top\mathbf r_N(\vartheta)$, one has
$\eta_N\in C(\Omega,\R^q)$ and
\[
\mathbf r_N(\vartheta)
=
D_N\eta_N(\vartheta),
\qquad
\|\eta_N(\vartheta)\|_2\le1,
\qquad
\vartheta\in\Omega.
\]
\end{proposition}

\begin{proof}
Since $F_N\mathbf r_N=0$ on $\Omega$,
$\mathbf r_N(\vartheta)\in\ker F_N$ for every
$\vartheta\in\Omega$. Hence $U_NU_N^\top\mathbf r_N(\vartheta)
=
\mathbf r_N(\vartheta),
\
\vartheta\in\Omega$.

Therefore, $D_N\eta_N(\vartheta)
=
U_NU_N^\top\mathbf r_N(\vartheta)
=
\mathbf r_N(\vartheta)$.
The assumed quadratic bound gives
$\|\eta_N(\vartheta)\|_2\le1$ on $\Omega$.
Continuity of $\eta_N$ follows from continuity of $\mathbf r_N$.
\end{proof}
If instead $\ker F_N=\{0\}$, then $\mathbf r_N=0$ on $\Omega$.
Exact linear constraints arise whenever $a\in\R^N$ satisfies
$a^\top\boldsymbol\gamma_N\in W_N$, since componentwise projection and the standing convention give
\[
a^\top\mathbf r_N
=
\res_{W_N}\bigl(a^\top\boldsymbol\gamma_N\bigr)
=
0
\quad\text{on }\Omega.
\]
Any finite collection of such row vectors $a^\top$ may be stacked to form $F_N$.
The choices $F_N=0$ and $U_N=I_N$ give a full-dimensional
ellipsoidal enclosure. For any $\bar\rho_N>0$ satisfying
$\rho_N^\star\le\bar\rho_N$, taking
$S_N=\bar\rho_N^2I_N$ recovers the scalar enclosure.

\section{Robust \texorpdfstring{$\mathcal{H}_\infty$}{H-infinity}
output-feedback synthesis and original-state bounds}
\label{sec:robust-hinf-uub}
\label{sec:control}
\subsection{Generalized plant and controller synthesis}
Fix $N$ with $1\le N\le M$ and set
$n_N\coloneqq\ell+N$. Using the exact lifted realization of
Theorem~\ref{thm:main-CCK}, we formulate a MIMO
$\mathcal{H}_\infty$ output-feedback problem with control input $u_a$ and
tracking signal $r-y$. Original-state estimates are subsequently
obtained from the exact readout $H_{\xi,N}$ whenever the corresponding
state-coordinate retention condition holds.

\begin{assumption}
\label{ass:feedback-setting}
\begin{enumerate}[label=\textnormal{(\roman*)}]
\item
$\Omega\subset P\times X$ has nonempty interior, and $f$ and $h$ are
locally Lipschitz on their respective open domains.

\item
The reference $r$ and actuator-realization defect $\delta_p$ are
essentially bounded measurable signals on $[0,\infty)$, with finite
bounds $\bar r,\delta\ge0$ satisfying
\[
\|r(t)\|_2\le\bar r,\\
\|\delta_p(t)\|_2\le\delta
\quad\text{a.e.}
\]
\end{enumerate}
\end{assumption}
To match the standard generalized-plant partition and separate the
two uncertainty blocks in \eqref{eq:Delta}, define
\begin{equation}
\begin{aligned}
E_{p,N}
\coloneqq
\begin{bmatrix}
I_\ell\\
0_{N\times\ell}
\end{bmatrix},
E_{c,N}
\coloneqq
\begin{bmatrix}
0_{\ell\times N}\\
I_N
\end{bmatrix}.
\label{eq:channel-matrices-s6}
\end{aligned}
\end{equation}
Then, along every admissible trajectory,
\[
\Delta_N
=
E_{p,N}\delta_p
+
E_{c,N}\bigl(\mathbf r_N+\widehat J_N\delta_p\bigr).
\]

For the defect bound $\delta$ in Assumption~\ref{ass:feedback-setting},
write $\sigma_{c,N}^{\star}:=\sigma_{c,N}^{\star}(\delta)$ as defined in
\eqref{eq:sigma-c-s6}.

Throughout this section, a trajectory is understood to be locally
absolutely continuous and to satisfy the governing differential
equations almost everywhere.

Consider a trajectory of
\eqref{eq:actuator} and \eqref{eq:plant-augmented}
on its maximal interval of existence $[0,T_{\max,N})$, with augmented state
$\xi(t)=\col\bigl(p(t),x(t)\bigr)$.
If $\xi(0)\in\operatorname{int}\Omega$, define the first exit time
\begin{equation}
\begin{aligned}
\tau_{\Omega,N}
\coloneqq
\inf
\left\{
t\in[0,T_{\max,N}):
\xi(t)\notin\operatorname{int}\Omega
\right\},
\label{eq:first-exit-s6}
\end{aligned}
\end{equation}
with $\inf\varnothing\coloneqq T_{\max,N}$.
\begin{proposition}
\label{prop:outer-realization}
Under Assumption~\ref{ass:feedback-setting}, let
$(p,x,u_a)$ be a trajectory of
\eqref{eq:actuator} and \eqref{eq:plant-augmented}, with
$\tau_{\Omega,N}$ defined by \eqref{eq:first-exit-s6}. Then, on
$[0,\tau_{\Omega,N})$, the lifted trajectory
\[
\zeta_N(t)
\coloneqq
\mathbf Z_N\bigl(p(t),x(t)\bigr)
\]
satisfies the LTI dynamics
\begin{align}
\dot\zeta_N
&=
A_N\zeta_N+B_Nu_a
+\delta E_{p,N}w_{p,N}
+\sigma_{c,N}^{\star}E_{c,N}w_{c,N}
\quad\text{a.e.},
\notag\\
y
&=
C_{y,N}\zeta_N+\nu_N^{\star}w_{v,N},
\label{eq:outer-realization}
\end{align}
for measurable channels $w_{p,N}$, $w_{c,N}$, and $w_{v,N}$,
representing the normalized actuator defect, lower-block uncertainty,
and output-readout residual, respectively, and taking values in
$\R^\ell$, $\R^N$, and $\R^{n_y}$,
\begin{equation}
\begin{aligned}
\|w_{p,N}(t)\|_2\le1,\
\|w_{c,N}(t)\|_2\le1,\
\|w_{v,N}(t)\|_2\le1
\label{eq:normalized-bounds-s6}
\end{aligned}
\end{equation}
a.e. on $[0,\tau_{\Omega,N})$.
\end{proposition}
\begin{proof}
Fix a compact interval
$J\subset[0,\tau_{\Omega,N})$. By the definition of the first exit
time,
$(p(t),x(t))\in\Omega
\text{for every }t\in J$.
The components of $\mathbf Z_N$ admit $C^1$ extensions to the open
neighborhood $\mathcal O$ of $\Omega$. Hence, along the locally
absolutely continuous trajectory,
\[
\begin{aligned}
\frac{\mathrm d}{\mathrm dt}
\mathbf Z_N(p(t),x(t))
=&
D_p\mathbf Z_N(p(t),x(t))\dot p(t)\\
& +D_x\mathbf Z_N(p(t),x(t))\dot x(t)
\end{aligned}
\]
for almost every $t\in J$.

Using the dynamics
\eqref{eq:actuator}--\eqref{eq:plant-augmented}, the field identity
\eqref{eq:Upsilon-split}, and
Lemma~\ref{lem:lower-block-exactB} gives
\begin{equation}
\begin{aligned}
\dot\zeta_N
={}&
A_N\zeta_N+B_Nu_a
+E_{p,N}\delta_p\\
&+
E_{c,N}
\bigl(
\mathbf r_N(p,x)
+\widehat J_N(p,x)\delta_p
\bigr)
\end{aligned}
\label{eq:unscaled-ae-s6}
\end{equation}
a.e. on $J$. For each positive channel scale, define
\[
\begin{aligned}
w_{p,N}(t)&=\delta^{-1}\delta_p,
\quad
w_{v,N}(t)=(\nu_N^\star)^{-1}v_{y,N}(p,x),\\
w_{c,N}(t)&=(\sigma_{c,N}^\star)^{-1}
\bigl[\mathbf r_N(p,x)+\widehat J_N(p,x)\delta_p\bigr].
\end{aligned}
\]
These signals are measurable. The $\delta$ in
Assumption~\ref{ass:feedback-setting}.(ii), with
\eqref{eq:sigma-c-s6} and \eqref{eq:nu-star}, yields
\eqref{eq:normalized-bounds-s6}.

Substitution into \eqref{eq:unscaled-ae-s6} gives the state equation in
\eqref{eq:outer-realization}. Moreover,
\eqref{eq:cck-output-readout-residual} evaluated along the trajectory
gives
\[
y(t)
=
C_{y,N}\zeta_N(t)
+
v_{y,N}(p(t),x(t)),
\]
which yields the output equation in \eqref{eq:outer-realization}. Since $J$ was arbitrary, the
state identity holds almost everywhere on
$[0,\tau_{\Omega,N})$, and the output identity holds throughout that
interval.
\end{proof}

Zero-scale channels are set to zero and omitted. The output-residual channel vanishes when every component of $h$
belongs to $W_N$. Otherwise, it records the normalized output-projection error
along the trajectory. The actual channels are coupled through the
trajectory and actuator defect. Treating them as independent unit-ball
inputs can make the resulting bounds conservative.

For $\delta=0$, Proposition~\ref{prop:shaped-residual}
allows $\sigma_{c,N}^{\star}E_{c,N}$ in
\eqref{eq:outer-realization} to be replaced by $E_{c,N}D_N$,
with $w_{c,N}(t)=\eta_N(\xi(t))\in\R^q$ and
$\|w_{c,N}(t)\|_2\le1$ before $\tau_{\Omega,N}$.
For $\delta>0$, the shaped enclosure and its exact constraints
must hold for $\mathbf r_N(\vartheta)+\widehat J_N(\vartheta)d$
on $\Omega\times\{d:\|d\|_2\le\delta\}$.

The exact outer model is augmented with fixed finite-dimensional
stable proper MIMO weighting filters $W_P$, $W_U$, and $W_T$, acting
on the tracking error, control input, and physical output,
respectively. The overall mixed-sensitivity structure may be selected
following standard MIMO robust-control practice
\cite{skogestad2005multivariable,kurkccu2022existence}.

Choose $\sigma_{r,N}>0$ and set
$w_{r,N}\coloneqq\sigma_{r,N}^{-1}r$.
The exogenous input is $w_N=\col(w_{r,N},w_{p,N},w_{c,N},w_{v,N})$
and the measured tracking signal is $y_{K,N}=e_N=r-y$.
Let
\begin{equation}
\begin{aligned}
z_N
\coloneqq
\col\bigl(W_Pe_N,\;W_Uu_a,\;W_Ty\bigr)
\label{eq:physical-performance-s6}
\end{aligned}
\end{equation}
denote the stacked outputs of the weighting filters.

Let $G_N$ be a finite-dimensional state-space realization of the
interconnection of \eqref{eq:outer-realization} with these filters:
\begin{equation}
\begin{bmatrix}z_N\\y_{K,N}\end{bmatrix}
=
\underbrace{\left[
\begin{array}{c|cc}
A_{G,N}&B_{1,N}&B_{2,N}\\ \hline
C_{1,N}&D_{11,N}&D_{12,N}\\
C_{2,N}&D_{21,N}&0
\end{array}
\right]}_{=:G_N}
\begin{bmatrix}w_N\\u_a\end{bmatrix}.
\label{eq:generalized-plant-s6}
\end{equation}
The generalized-plant state is $x_{G,N}=\col(\zeta_N,x_W)$,
where $x_W$ collects the weighting-filter states. In the above channel
order, the feedthrough blocks are
\begin{equation}
\begin{aligned}
D_{12,N}
=
\col\bigl(0,W_U(\infty),0\bigr),
\\
D_{21,N}
=
\begin{bmatrix}
\sigma_{r,N}I_{n_y}&0&0&-\nu_N^\star I_{n_y}
\end{bmatrix}.
\label{eq:generalized-feedthrough-s6}
\end{aligned}
\end{equation}
The remaining realization matrices are obtained by standard state
augmentation.
For any finite-dimensional proper controller $K_N$, the lower
interconnection with $G_N$ is well posed because $D_{22,N}=0$.
Its zero-initial-state map from $w_N$ to $z_N$ is
$\mathcal F_\ell(G_N,K_N)$.
Let $x_{K,N}$ be the controller state and set
$\chi_N=\col(\zeta_N,x_W,x_{K,N})$.

Whenever \eqref{eq:state-coordinate-retention} holds, write
\[
C_{\xi,N}:=\begin{bmatrix}H_{\xi,N}&0&0\end{bmatrix},
\qquad \xi=C_{\xi,N}\chi_N.
\]
for the physical-state readout in the above canonical state ordering.

Let $n_{\chi,N}$ denote the dimension of $\chi_N$. Define $C_{e,N}
\coloneqq
\begin{bmatrix}
C_{2,N}&0
\end{bmatrix},
D_{e,N}
\coloneqq
D_{21,N}$,
where the zero block corresponds to the controller state. The
closed-loop realization is
\begin{subequations}
\label{eq:closed-loop-realization-s6}
\begin{align}
\dot\chi_N
&=
A_{\mathrm{cl},N}\chi_N+B_{\mathrm{cl},N}w_N,
\label{eq:closed-state-s6}\\
z_N
&=
C_{\mathrm{cl},N}\chi_N+D_{\mathrm{cl},N}w_N,
\label{eq:closed-output-s6}\\
e_N
&=
C_{e,N}\chi_N+D_{e,N}w_N.
\label{eq:closed-error-s6}
\end{align}
\end{subequations}
where $A_{\mathrm{cl},N},\ B_{\mathrm{cl},N},\ C_{\mathrm{cl},N},$ and $D_{\mathrm{cl},N}$ are obtained by setting $D_{22,N}=0$ in 
\cite[Sec.~9.3, p.~180]{ZhouDoyle1998}.

By Proposition~\ref{prop:outer-realization} and the state augmentation
used to construct $G_N$, every nonlinear trajectory covered by that
proposition, together with the corresponding weighting-filter and
controller states, satisfies
\eqref{eq:closed-loop-realization-s6} almost everywhere before its
first exit from $\Omega$.

\begin{assumption}
\label{ass:hinf-feasibility}

For the selected order $N$, $(A_{G,N},B_{2,N})$ is stabilizable and
$(A_{G,N},C_{2,N})$ is detectable. The standard regularity conditions~\cite{DGKF1989,Safonov1989LoopShifting} for continuous-time $\mathcal H_\infty$ synthesis hold at each tested performance level.

\end{assumption}
Let $\mathcal{RH}_\infty$ denote the set of real-rational proper
transfer matrices whose poles lie in the open left half-plane. For
$T\in\mathcal{RH}_\infty$, define
\[
\|T\|_\infty
\coloneqq
\sup_{\omega\in\R}
\bar\sigma\bigl(T(\mathrm j\omega)\bigr),
\]
where $\bar\sigma$ denotes the largest singular value. Moreover, let $\mathcal K_{\mathrm{stab},N}$ denote the finite-dimensional proper
controllers that internally stabilize the lower interconnection with
$G_N$, and define
\begin{equation}
\begin{aligned}
\gamma_{N,\mathrm{syn}}^\star
\coloneqq
\inf_{K_N\in\mathcal K_{\mathrm{stab},N}}
\left\|
\mathcal F_\ell(G_N,K_N)
\right\|_\infty.
\label{eq:synthesis-problem-s6}
\end{aligned}
\end{equation}

For fixed $G_N$, stabilizability and detectability ensure that
$\mathcal K_{\mathrm{stab},N}$ is nonempty and
$\gamma_{N,\mathrm{syn}}^\star$ is finite.
With controller order and structure unrestricted, this problem can also be written as an equivalent convex model-matching problem using the Youla parameterization~\cite{anderson2019systemlevel,kurkcutac26}.

Under Assumption~\ref{ass:hinf-feasibility}, apply standard
normalization and, where needed, loop shifting at a tested level
$\gamma_{N,\mathrm{syn}}>0$~\cite{Safonov1989LoopShifting}.
If the Riccati feasibility conditions hold at this level, the
Doyle--Glover--Khargonekar--Francis construction~\cite{DGKF1989},
transformed back to the original generalized plant, yields a
finite-dimensional proper internally stabilizing controller $K_N$
satisfying
\begin{equation}
\begin{aligned}
\left\|
\mathcal F_\ell(G_N,K_N)
\right\|_\infty
<
\gamma_{N,\mathrm{syn}}.
\label{eq:synthesis-level-s6}
\end{aligned}
\end{equation}

The closure-channel matrix $E_{c,N}D_N$ enters $G_N$ through its
external-input columns, leaving the nominal feedback interconnection
and its closed-loop state matrix unchanged for a fixed controller.

We compare the scalar closure channel in~\eqref{eq:outer-realization} with its shaped replacement $E_{c,N}D_N$. Suppose
$D_N=\sigma_{c,N}^\star M_{c,N}$ with
$\|M_{c,N}\|_{2\to2}\le1$.
Fix all other channels and weights.
Let $T_{zw,N}^{\mathrm{scalar}}$ and $T_{zw,N}^{D}$ denote the
closed-loop maps for the scalar and shaped closure channels, respectively.
Their internally stabilizing controllers form the same class, and every
such $K_N$ satisfies
\[
\begin{aligned}
T_{zw,N}^{D}(s;K_N)
&=T_{zw,N}^{\mathrm{scalar}}(s;K_N)\mathsf M_N,\\
\mathsf M_N&:=\operatorname{diag}(I_{n_y+\ell},M_{c,N},I_{n_y}).
\end{aligned}
\]
in the ordering $(w_{r,N},w_{p,N},w_{c,N},w_{v,N})$,
with zero-scale blocks omitted.
Thus $\|T_{zw,N}^{D}\|_\infty\le
\|T_{zw,N}^{\mathrm{scalar}}\|_\infty$, also after taking the infimum
over these controllers. The same contraction preserves any
quadratic bounded-real dissipation inequality of the form
\eqref{eq:dissipation-s6}: if $P_N$, $\varepsilon_N$, and
$\gamma_{N,\mathrm a}$ certify the scalar-radius realization, they also
certify the shaped realization, since
$w_N^{\mathrm{scalar}}=\mathsf M_Nw_N^D$ and
$\|\mathsf M_Nw_N^D\|_2\le\|w_N^D\|_2$.

Since $\mathcal F_\ell(G_N,K_N)$ is the zero-initial-state map from
$w_N$ to $z_N$, \eqref{eq:synthesis-level-s6} implies
\begin{equation}
\begin{aligned}
\|z_N\|_{L^2}
<
\gamma_{N,\mathrm{syn}}
\|w_N\|_{L^2}
\label{eq:induced-L2-s6}
\end{aligned}
\end{equation}
for every nonzero $w_N\in L^2$ with $\chi_N(0)=0$.

The induced-$L^2$ synthesis guarantee is complemented below by a
quadratic bounded-input estimate for the fixed closed loop.

\subsection{Closed-loop analysis}

Fix an internally stabilizing controller $K_N$ satisfying
\eqref{eq:synthesis-level-s6}. We derive a quadratic estimate for the
closed-loop realization \eqref{eq:closed-loop-realization-s6}.

Define the closed-loop transfer matrix
\begin{equation}
\begin{aligned}
T_{zw,N}(s)
\coloneqq
C_{\mathrm{cl},N}
(sI-A_{\mathrm{cl},N})^{-1}
B_{\mathrm{cl},N}
+
D_{\mathrm{cl},N}.
\label{eq:Tzw-s6}
\end{aligned}
\end{equation}
Thus $T_{zw,N}=\mathcal F_\ell(G_N,K_N)$.

The analysis may use any fixed similarity transformation
$\widetilde\chi_N=T_{\chi,N}\chi_N$, with
\[
\begin{aligned}
\widetilde A_{\mathrm{cl},N}&=T_{\chi,N}A_{\mathrm{cl},N}T_{\chi,N}^{-1},\\
\widetilde B_{\mathrm{cl},N}&=T_{\chi,N}B_{\mathrm{cl},N}.
\end{aligned}
\]
Each output or physical-state readout $C$ becomes $CT_{\chi,N}^{-1}$.
Direct feedthrough is unchanged. For notational simplicity, we omit tildes below. Each certificate uses one fixed realization.

The following consequence of the classical bounded-real lemma
provides the quadratic estimate used in the original-state bounds.
\begin{theorem}
\label{thm:lyapunov-are}
For any $\gamma_{N,\mathrm a}
>
\|T_{zw,N}\|_\infty$, there exist
$P_N=P_N^\top\succ0$ and $\varepsilon_N>0$ such that
\begin{equation}
\begin{aligned}
V_{\mathrm{cl},N}(\chi_N)
\coloneqq
\chi_N^\top P_N\chi_N
\label{eq:lyapunov-function-s6}
\end{aligned}
\end{equation}
satisfies
\begin{equation}
\begin{aligned}
\dot V_{\mathrm{cl},N}
+
\|z_N\|_2^2
+
\varepsilon_N\|\chi_N\|_2^2
-
\gamma_{N,\mathrm a}^2\|w_N\|_2^2
\le0 \quad
\text{a.e.}
\label{eq:dissipation-s6}
\end{aligned}
\end{equation}
along \eqref{eq:closed-loop-realization-s6}.
\end{theorem}
\begin{proof}
Internal stability implies that $A_{\mathrm{cl},N}$ is Hurwitz. Define $T_{\chi w,N}(s)
\coloneqq
(sI-A_{\mathrm{cl},N})^{-1}B_{\mathrm{cl},N}$.

Then $T_{\chi w,N}\in\mathcal{RH}_\infty$. For every
$\varepsilon>0$,
\[
\left\|
\begin{bmatrix}
T_{zw,N}\\
\sqrt{\varepsilon}\,T_{\chi w,N}
\end{bmatrix}
\right\|_\infty^2
\le
\|T_{zw,N}\|_\infty^2
+
\varepsilon\|T_{\chi w,N}\|_\infty^2.
\]
If $T_{\chi w,N}\ne0$, choose
\[
0<\varepsilon_N
<
\frac{
\gamma_{N,\mathrm a}^2-\|T_{zw,N}\|_\infty^2
}{
\|T_{\chi w,N}\|_\infty^2
};
\]
if $T_{\chi w,N}=0$, choose any $\varepsilon_N>0$. Hence the transfer
from $w_N$ to $z_{\varepsilon,N}
\coloneqq
\col\bigl(z_N,\sqrt{\varepsilon_N}\chi_N\bigr)$ has $\mathcal{H}_\infty$ norm strictly below $\gamma_{N,\mathrm a}$.

Since $D_{\mathrm{cl},N}$ is the feedthrough matrix of $T_{zw,N}$,
\[
\|D_{\mathrm{cl},N}\|_{2\to2}
\le
\|T_{zw,N}\|_\infty
<
\gamma_{N,\mathrm a}.
\]
Therefore,
\[
R_{\gamma,N}
\coloneqq
\gamma_{N,\mathrm a}^2I
-
D_{\mathrm{cl},N}^\top D_{\mathrm{cl},N}
\succ0.
\]
Because $A_{\mathrm{cl},N}$ is Hurwitz,
$(A_{\mathrm{cl},N},B_{\mathrm{cl},N})$ is stabilizable; the augmented
output contains $\sqrt{\varepsilon_N}I$, so detectability is
automatic. The continuous-time strict bounded-real lemma in
Riccati form~\cite[Ch.~12]{ZhouDoyle1998} therefore gives a stabilizing solution
$P_N=P_N^\top\succeq0$ of
\begin{equation}
\begin{aligned}
0={}&
A_{\mathrm{cl},N}^\top P_N
+
P_NA_{\mathrm{cl},N}
+
C_{\mathrm{cl},N}^\top C_{\mathrm{cl},N}
+
\varepsilon_NI\\
&+
\Theta_NR_{\gamma,N}^{-1}\Theta_N^\top,
\end{aligned}
\label{eq:analysis-ARE-s6}
\end{equation}
where $\Theta_N
\coloneqq
P_NB_{\mathrm{cl},N}
+
C_{\mathrm{cl},N}^\top D_{\mathrm{cl},N}$.
Equation~\eqref{eq:analysis-ARE-s6} is a single post-synthesis Riccati
equation for the fixed closed loop.

To verify positive definiteness, let $q\in\ker P_N$. Premultiplying
and postmultiplying \eqref{eq:analysis-ARE-s6} by $q^\top$ and $q$
gives
\[
0
=
\|C_{\mathrm{cl},N}q\|_2^2
+
\varepsilon_N\|q\|_2^2
+
\left\|
R_{\gamma,N}^{-1/2}
D_{\mathrm{cl},N}^\top C_{\mathrm{cl},N}q
\right\|_2^2.
\]
Thus $q=0$, and hence $P_N\succ0$.

Finally, substituting \eqref{eq:analysis-ARE-s6} into the derivative of
\eqref{eq:lyapunov-function-s6} along
\eqref{eq:closed-state-s6} and completing the square yields
\begin{align}
&\dot V_{\mathrm{cl},N}
+
\|z_N\|_2^2
+
\varepsilon_N\|\chi_N\|_2^2
-
\gamma_{N,\mathrm a}^2\|w_N\|_2^2
\notag\\
&\quad=
-\left\|R_{\gamma,N}^{1/2}
\bigl(w_N-R_{\gamma,N}^{-1}\Theta_N^\top\chi_N\bigr)
\right\|_2^2\le0,
\label{eq:square-completion-s6}
\end{align}
which proves \eqref{eq:dissipation-s6}.
\end{proof}
\subsection{Original-state bounds}
For each nonlinear trajectory covered by
Proposition~\ref{prop:outer-realization}, the normalized exogenous
input $w_N$ is essentially bounded on $[0,\tau_{\Omega,N})$ by
Assumption~\ref{ass:feedback-setting}, the definition of $w_{r,N}$,
and \eqref{eq:normalized-bounds-s6}. Define
\begin{equation}
\begin{aligned}
m_{w,N}
\coloneqq
\operatorname*{ess\,sup}_{0\le t<\tau_{\Omega,N}}
\|w_N(t)\|_2.
\label{eq:muN-s6}
\end{aligned}
\end{equation}
Along a closed-loop trajectory, abbreviate
$V_{\mathrm{cl},N}(t)
\coloneqq V_{\mathrm{cl},N}(\chi_N(t))$.

\begin{corollary}
\label{cor:lyapunov-estimates}
Under the hypotheses of Theorem~\ref{thm:lyapunov-are}, define
\begin{equation}
\begin{aligned}
a_N
\coloneqq
\frac{\varepsilon_N}{\lambda_{\max}(P_N)}
>0.
\label{eq:aN-s6}
\end{aligned}
\end{equation}
Then
\begin{equation}
\begin{aligned}
V_{\mathrm{cl},N}(t)
\le
e^{-a_Nt}V_{\mathrm{cl},N}(0)
+
\frac{\gamma_{N,\mathrm a}^2}{a_N}
\bigl(1-e^{-a_Nt}\bigr)m_{w,N}^2
\label{eq:V-bound-s6}
\end{aligned}
\end{equation}
for every $0\le t<\tau_{\Omega,N}$.
\end{corollary}
\begin{proof}
Dropping $\|z_N\|_2^2$ from \eqref{eq:dissipation-s6} and using $V_{\mathrm{cl},N}
\le
\lambda_{\max}(P_N)\|\chi_N\|_2^2$ gives
\[
\dot V_{\mathrm{cl},N}
\le
-a_NV_{\mathrm{cl},N}
+
\gamma_{N,\mathrm a}^2m_{w,N}^2
\text{a.e. on }[0,\tau_{\Omega,N}).
\]
Multiplying by $e^{a_Nt}$ and integrating from $0$ to $t$ gives
\eqref{eq:V-bound-s6}.
\end{proof}

At any order $N$ satisfying the state-coordinate retention condition
\eqref{eq:state-coordinate-retention}, define the physical-state and
tracking-error readout gains
\begin{equation}
\begin{aligned}
\kappa_{\xi,N}
\coloneqq
\left\|
C_{\xi,N}P_N^{-1/2}
\right\|_{2\to2},
\\
\kappa_{e,N}
\coloneqq
\|C_{e,N}P_N^{-1/2}\|_{2\to2},
\label{eq:kappas-s6}
\end{aligned}
\end{equation}

\begin{theorem}
\label{thm:preexit-uub}
Under the hypotheses of Theorem~\ref{thm:lyapunov-are}, at any order
$N$ satisfying \eqref{eq:state-coordinate-retention}, every nonlinear
closed-loop trajectory covered by Proposition~\ref{prop:outer-realization}
satisfies
\begin{equation}
\begin{aligned}
\|\xi(t)\|_2
\le
\kappa_{\xi,N}
\left[\begin{gathered}
e^{-a_Nt}V_{\mathrm{cl},N}(0)\\
+
\frac{\gamma_{N,\mathrm a}^2}{a_N}
\bigl(1-e^{-a_Nt}\bigr)m_{w,N}^2
\end{gathered}\right]^{1/2}
\label{eq:preexit-state-s6}
\end{aligned}
\end{equation}
for every $0\le t<\tau_{\Omega,N}$. Moreover,
\begin{equation}
\begin{aligned}
&\|r(t)-h(x(t))\|_2\\
&\quad\le
\kappa_{e,N}
\left[
e^{-a_Nt}V_{\mathrm{cl},N}(0)
+
\frac{\gamma_{N,\mathrm a}^2}{a_N}
\bigl(1-e^{-a_Nt}\bigr)m_{w,N}^2
\right]^{1/2}
\\
&\quad +
\|D_{e,N}\|_{2\to2}m_{w,N}
\label{eq:preexit-error-s6}
\end{aligned}
\end{equation}
for almost every $0\le t<\tau_{\Omega,N}$.
\end{theorem}

\begin{proof}
Apply the closed-loop embedding above and
Corollary~\ref{cor:lyapunov-estimates}.

Under \eqref{eq:state-coordinate-retention},
Corollary~\ref{cor:cck-readouts}\textnormal{(ii)} gives
$\xi(t)=H_{\xi,N}\zeta_N(t)=C_{\xi,N}\chi_N(t)$
for every $0\le t<\tau_{\Omega,N}$. Since $P_N\succ0$,
\[
\begin{aligned}
\|\xi(t)\|_2
&\le
\left\|
C_{\xi,N}P_N^{-1/2}
\right\|_{2\to2}
\|P_N^{1/2}\chi_N(t)\|_2\\
&=
\kappa_{\xi,N}
\sqrt{V_{\mathrm{cl},N}(t)}.
\end{aligned}
\]
Combining this inequality with \eqref{eq:V-bound-s6} proves
\eqref{eq:preexit-state-s6}.

Likewise, \eqref{eq:closed-error-s6} gives, for almost every
$t<\tau_{\Omega,N}$,
\[
\begin{aligned}
\|e_N(t)\|_2
&\le
\|C_{e,N}P_N^{-1/2}\|_{2\to2}
\|P_N^{1/2}\chi_N(t)\|_2
\\
&\quad +
\|D_{e,N}\|_{2\to2}\|w_N(t)\|_2\\
&\le
\kappa_{e,N}
\sqrt{V_{\mathrm{cl},N}(t)}
+
\|D_{e,N}\|_{2\to2}m_{w,N}.
\end{aligned}
\]
Since $e_N=r-y=r-h(x)$, applying \eqref{eq:V-bound-s6} proves
\eqref{eq:preexit-error-s6}.
\end{proof}
With the zero-scale convention adopted above, define the uniform
outer-channel bound
\begin{equation}
\begin{aligned}
\bar m_{w,N}
\coloneqq
\left[
\left(
\frac{\bar r}{\sigma_{r,N}}
\right)^2
+
\mathbf 1_{\{\delta>0\}}
+
\mathbf 1_{\{\sigma_{c,N}^{\star}>0\}}
+
\mathbf 1_{\{\nu_N^{\star}>0\}}
\right]^{1/2},
\label{eq:muN-outer-s6}
\end{aligned}
\end{equation}
where $\mathbf 1_{\{\cdot\}}$ denotes the indicator. The definition of
$w_{r,N}$, Assumption~\ref{ass:feedback-setting}, and
\eqref{eq:normalized-bounds-s6} give $m_{w,N}\le\bar m_{w,N}$.
A nonzero shaped closure channel also contributes one.
We use three sufficient checks for regional containment.

\emph{(i) Performance-storage ellipsoid.} Define
\begin{equation}
\begin{aligned}
c_N
\coloneqq
\max
\left\{
V_{\mathrm{cl},N}(0),
\frac{\gamma_{N,\mathrm a}^2}{a_N}
\bigl(\bar m_{w,N}\bigr)^2
\right\},
\label{eq:cN-s6}
\end{aligned}
\end{equation}
the closed-loop ellipsoid
\begin{equation}
\begin{aligned}
\mathfrak E_{\chi,N}(c_N)
\coloneqq
\left\{
\chi\in\R^{n_{\chi,N}}:
\chi^\top P_N\chi\le c_N
\right\},
\label{eq:Echi-s6}
\end{aligned}
\end{equation}
and its physical-state image
\begin{equation}
\begin{aligned}
\mathfrak E_{\xi,N}(c_N)
\coloneqq
\left\{
C_{\xi,N}\chi:
\chi\in\mathfrak E_{\chi,N}(c_N)
\right\}.
\label{eq:Exi-s6}
\end{aligned}
\end{equation}

\emph{(ii) Auxiliary ellipsoid.} Omit the performance-filter
states, which do not enter the feedback loop, and optimize localization
separately. In the canonical interconnection,
\[
\begin{aligned}
\dot\chi_{f,N}
&=
A_{f,N}\chi_{f,N}+B_{f,N}w_N,
\\
\chi_{f,N}
&=\col(\zeta_N,x_{K,N}),
\\
\xi
&=H_{f,N}\chi_{f,N}.
\end{aligned}
\]
where $H_{f,N}=[H_{\xi,N}\ 0]$. A separate similarity transformation
may be applied to this realization and its readout.
For each retained vector channel $j\in\mathcal J_N$, let
$m_j=\dim w_{j,N}$ and $\|w_{j,N}\|_2\le b_{j,N}$, with
$b_{r,N}=\bar r/\sigma_{r,N}$ and $b_{j,N}=1$ for uncertainty channels.

Given $P_{\mathrm{loc},N}=P_{\mathrm{loc},N}^\top\succ0$,
$\alpha_N>0$, and $\beta_{j,N}>0$ satisfying
\[
\begin{bmatrix}
\begin{gathered}
A_{f,N}^\top P_{\mathrm{loc},N}+P_{\mathrm{loc},N}A_{f,N}\\
{}+\alpha_NP_{\mathrm{loc},N}
\end{gathered}
& P_{\mathrm{loc},N}B_{f,N}\\[4pt]
B_{f,N}^\top P_{\mathrm{loc},N} &
-\operatorname{diag}_{j\in\mathcal J_N}(\beta_{j,N}I_{m_j})
\end{bmatrix}\preceq0,
\]
define
\[
\begin{gathered}
V_{\mathrm{loc},N}(\chi_f)=\chi_f^\top P_{\mathrm{loc},N}\chi_f,\\[3pt]
\mathfrak E_{f,N}(c)=\{\chi_f:V_{\mathrm{loc},N}(\chi_f)\le c\},
\end{gathered}
\]
\[
c_{\mathrm{loc},N}=\max\left\{
\begin{gathered}
V_{\mathrm{loc},N}(\chi_{f,N}(0)),\\[3pt]
\frac{\sum_{j\in\mathcal J_N}\beta_{j,N}b_{j,N}^2}{\alpha_N}
\end{gathered}
\right\}.
\]

\emph{(iii) Convolution bounds.} The auxiliary state $\chi_{\mathrm{aux},N}$ realizes the physical
feedback dynamics with an exact readout of $\xi$. Section~\ref{sec:numerics} uses the lifted auxiliary state
$\chi_{\mathrm{aux},N}=\col(\mathbf Z_N(\xi),x_{K,N})$.
Suppose these dynamics admit, before first exit, the realization
\[
\begin{aligned}
\dot\chi_{\mathrm{aux},N}
 &=A_{\mathrm{aux},N}\chi_{\mathrm{aux},N}
   +B_{\mathrm{aux},r,N}r
   +\sum_{j=1}^{m_a}B_{\mathrm{aux},j,N}v_j,\\
\xi&=H_{\mathrm{aux},N}\chi_{\mathrm{aux},N},
\qquad \|v_j\|_2\le c_j,
\end{aligned}
\]
with $A_{\mathrm{aux},N}$ Hurwitz, $m_a<\infty$, $c_j\ge0$, and
consistent initialization. Let $h_i$ be row $i$ of $H_{\mathrm{aux},N}$.
For a reference class $\mathscr R$ satisfying
Assumption~\ref{ass:feedback-setting}, choose $q_i\ge0$ such that
\[
\sup_{r\in\mathscr R}\sup_{t\ge0}
\left|h_i\int_0^t e^{A_{\mathrm{aux},N}(t-s)}
 B_{\mathrm{aux},r,N}r(s)\,ds\right|\le q_i.
\]
For a box $\Omega=\{\xi:|\xi_i|\le b_i\}$, $b_i>0$, define
\[
\begin{aligned}
d_i={}&\sup_{t\ge0}|h_ie^{A_{\mathrm{aux},N}t}
                    \chi_{\mathrm{aux},N}(0)|+q_i\\
 &+\sum_{j=1}^{m_a}c_j\int_0^\infty
     \|h_ie^{A_{\mathrm{aux},N}t}B_{\mathrm{aux},j,N}\|_2\,dt.
\end{aligned}
\]
For an initial class, take the supremum of the homogeneous term and
initial storage values over that class.

\begin{theorem}
\label{thm:regional-completion}

Under the hypotheses of Theorem~\ref{thm:preexit-uub}, let
$\xi(0)\in\operatorname{int}\Omega$ and
$\zeta_N(0)=\mathbf Z_N(\xi(0))$. Suppose that
\begin{equation}
\begin{aligned}
\mathfrak E_{\xi,N}(c_N)\subset\operatorname{int}\Omega,
\label{eq:containment-s6}
\end{aligned}
\end{equation}
or that alternative \textnormal{(ii)} above satisfies
$H_{f,N}\mathfrak E_{f,N}(c_{\mathrm{loc},N})\subset\operatorname{int}\Omega$,
or that alternative \textnormal{(iii)} has $r\in\mathscr R$ and
$d_i<b_i$ for every physical coordinate.
Then the nonlinear closed-loop trajectory is forward complete and
\begin{equation}
\begin{aligned}
\xi(t)\in\mathfrak E_{\xi,N}(c_N)\cap\operatorname{int}\Omega
,\quad t\ge0.
\label{eq:regional-invariance-s6}
\end{aligned}
\end{equation}
The estimates of Theorem~\ref{thm:preexit-uub} hold for all time, with
\begin{align}
&\limsup_{t\to\infty}\|\xi(t)\|_2
\le\kappa_{\xi,N}\frac{\gamma_{N,\mathrm a}}{\sqrt{a_N}}
\bar m_{w,N},
\label{eq:state-UUB-s6}\\
&\operatorname*{ess\,limsup}_{t\to\infty}\|r(t)-h(x(t))\|_2
\notag\\*[-2pt]
&\quad\le\left(\kappa_{e,N}\frac{\gamma_{N,\mathrm a}}{\sqrt{a_N}}
+\|D_{e,N}\|_{2\to2}\right)\bar m_{w,N}.
\label{eq:error-UUB-s6}
\end{align}

\end{theorem}
\begin{proof}

Corollary~\ref{cor:lyapunov-estimates} and
$m_{w,N}\le\bar m_{w,N}$ give $V_{\mathrm{cl},N}(t)\le c_N$
before exit. Under \eqref{eq:containment-s6}, the physical state therefore
remains in a compact subset of $\operatorname{int}\Omega$.
For the auxiliary ellipsoid, the matrix inequality gives
\[
\dot V_{\mathrm{loc},N}\le-\alpha_NV_{\mathrm{loc},N}
+\sum_{j\in\mathcal J_N}\beta_{j,N}b_{j,N}^2.
\]
Integration yields $V_{\mathrm{loc},N}(t)\le c_{\mathrm{loc},N}$, so
$\xi(t)\in H_{f,N}\mathfrak E_{f,N}(c_{\mathrm{loc},N})
\subset\operatorname{int}\Omega$.

For the convolution alternative, variation of constants gives
$|\xi_i(t)|\le d_i$ before exit. The physical state therefore remains
in the compact box $\{\xi:|\xi_i|\le d_i\}\subset\operatorname{int}\Omega$.
In each case, continuity excludes a first exit before $T_{\max,N}$.
Moreover, $P_N\succ0$ and $V_{\mathrm{cl},N}\le c_N$ bound the entire
closed-loop state. The physical state remains in a compact subset of
its open domain, and the exogenous inputs are essentially bounded.
The continuation theorem for locally Lipschitz state equations therefore
gives $T_{\max,N}=\infty$.

The pre-exit estimates now hold for all time. Taking their upper limits
proves \eqref{eq:state-UUB-s6}--\eqref{eq:error-UUB-s6}.
Since $V_{\mathrm{cl},N}\ge\lambda_{\min}(P_N)\|\chi_N\|_2^2$,
the complete state also satisfies
$\limsup_{t\to\infty}\|\chi_N(t)\|_2
\le\gamma_{N,\mathrm a}\bar m_{w,N}/
\sqrt{a_N\lambda_{\min}(P_N)}$.
\end{proof}
\section{Numerical example}
\label{sec:numerics}
We consider the actuator-driven cart--pendulum\footnote{The code used in this paper is available at
\url{https://github.com/kurkcub/RCCG}.} based on the
mechanical model in~\cite[Sec.~3.2.1]{bruder2025koopman},
\begin{equation}
\begin{aligned}
\dot p_1&=p_2,&
\dot p_2&=-10^4p_1-200p_2+10^4u_a,\\
\dot x_1&=x_2,&
\dot x_2&=a(x)+b(x_1)p_1,\qquad y=x_1,
\end{aligned}
\label{eq:numerical-plant}
\end{equation}
with $u=p_1$, $\delta_p=0$, and
\[
\begin{aligned}
a(x)&=\frac{-9.81\sin x_1-x_2/0.9
-0.45\cos x_1\sin x_1\,x_2^2}{1-0.45\cos^2x_1},\\
b(x_1)&=\frac{-0.5\cos x_1}{1-0.45\cos^2x_1}.
\end{aligned}
\]
Here $x_1$ is the pendulum angle from the downward vertical
in radians, $x_2$ is its angular velocity in rad/s, and $p_1$ is
the horizontal cart force. The cart and pendulum masses are $1.1$
and $0.9$ kg, the length is $L_{\rm pend}=1$ m, and the rotational
damping coefficient is $1$ N\,m\,s.

We first examine model accuracy as the dictionary degree
increases. We then establish regional closed-loop bounds and compare
Galerkin and Jacobian $\mathcal H_\infty$ feedback, backstepping, and
Koopman-based hybrid MPC.

\subsection{Realization, residual, and dictionary degree}
The Galerkin model uses Lebesgue measure
$d\mu(\vartheta)=dp_1\,dp_2\,dx_1\,dx_2$ on
\[
\Omega=[-100,100]\times[-10^5,10^5]
\times[-0.75,0.75]\times[-0.5,0.5]
\]
in the order $\xi=\col(p_1,p_2,x_1,x_2)$. Polynomials, sine and cosine have continuous derivatives of every order. The expressions for $a(x)$ and $b(x_1)$ combine these functions with division by $1-0.45\cos^2x_1\ge0.55>0$. Therefore the vector field in~\eqref{eq:numerical-plant} has continuous derivatives of every order on a neighborhood of $\Omega$. The actuator matrices satisfy $C_pB_p=0$ for
$C_p=[1\ 0]$ and $B_p=[0\ 10^4]^\top$.
With the physicists' Hermite polynomials, the raw observables are
\begin{equation}
\Phi_{\mathrm{raw}}=\tfrac12\col\!\left(
\begin{gathered}
H_1(x_1),\ H_1(x_2),\ H_3(x_2),\\
H_1(x_1)H_2(x_2),\ H_2(x_1)H_1(x_2),\ H_3(x_1)
\end{gathered}\right).
\label{eq:numerical-dictionary}
\end{equation}
Thus $N=6$ and $n_N=8$. Multilevel dimension iteration for tensor products
(MDI-TP)~\cite{FengZhong2022MDI} evaluates the integrals in
$A_NR_N=Q_N$ using Gauss--Legendre rules of orders $(4,4,64,12)$.
The actuator and input blocks are retained exactly. The Gram factorization
$\int_\Omega\Phi_{\mathrm{raw}}\Phi_{\mathrm{raw}}^\top d\vartheta=L_{\rm Gram}L_{\rm Gram}^\top$
relates these coordinates to $\mathbf\Phi_N=L_{\rm Gram}^{-1}\Phi_{\mathrm{raw}}$.
Numerical matrices and residuals are expressed in raw coordinates.
The controller and performance certificates use the stored quadrature
matrix and its directly evaluated residual.
The first two modes give the plant-state readout exactly, so
$\nu_N^\star=0$.

We evaluate $\mathbf r_N=\boldsymbol\gamma_N-(A_N\mathbf Z_N)_{3:8}$.
The exact identities $\dot\phi_1^{\mathrm{raw}}=\phi_2^{\mathrm{raw}}$
and $\dot\phi_6^{\mathrm{raw}}=3\phi_5^{\mathrm{raw}}$ give
$F_N=[e_1\ e_6]^\top$ and $U_N=[e_2\ e_3\ e_4\ e_5]$.
For $D_N=U_NR_D$, a nonsingular factor is
\[
R_D\simeq\begin{bmatrix}
17.9590&4.73624&-0.101303&-23.0284\\
-280.872&305.299&0.837300&190.546\\
1.52533&-16.7924&-317.824&0.00106679\\
-71.7962&-22.2351&0.467931&106.407
\end{bmatrix}.
\]
Interval arithmetic using its full-precision coefficients verifies
\[
\sup_{\vartheta\in\Omega}
\|R_D^{-1}U_N^\top\mathbf r_N(\vartheta)\|_2\le0.99999<1.
\]
This realizes Proposition~\ref{prop:shaped-residual} with
$S_N=R_DR_D^\top$, up to an orthogonal channel transformation.
Affinity in $p_1$ and independence of $p_2$ reduce the verification
to the actuator endpoints and the plant-state rectangle.

Figure~\ref{fig:wide-realization}(a) uses
$u_a(t)=12\sin(0.2t)+2\sin(0.47t)$ and $\xi(0)=(0,0,0,0.1)$
over $40$ s. The path remains in $\Omega$. With the actual residual, the errors in $x_1$ and $x_2$ are within numerical precision, and the normalized residual peaks at $0.0952$.
Panel (b) uses a $17\times17$ grid over
$x_1(0)\in[-0.7,0.7]$, $x_2(0)\in[-0.45,0.45]$, with
$p_1(0)=-19.62\tan x_1(0)$, $p_2(0)=0$, and
$u_a(t)=p_1(0)+2\sin(0.2t)$. The simulations in panel (b) use RK4 with a $2.5\times10^{-4}$ s step.
Paths contribute until first exit. Of 289 paths, 268 remain at
$10$ s, with largest normalized residual $0.1878$.

Degrees $d=1,3,5$ give $4,8,14$ lifted states under the same
region, measure, command, and initialization as panel (a).
Peak angle errors are $(0.136,0.0522,0.0512)$ rad, with velocity
errors $(0.0688,0.0207,0.0205)$ rad/s. The dictionary uses state
polynomials to approximate $a(x)$ and retains actuator coordinates separately.
The coupling $b(x_1)p_1$ remains outside this family, explaining
the smaller change from degree three to five. Products of actuator
and state observables provide a further enrichment. We use $d=3$ for controller synthesis.

\begin{figure*}[t]
\centering
\begin{subfigure}[b]{0.35\textwidth}
\centering
\includegraphics[width=\linewidth]{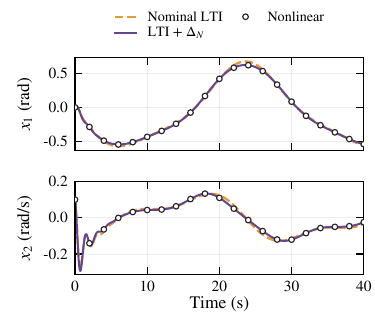}
\caption{Nominal prediction and residual correction.}
\end{subfigure}\hfill
\begin{subfigure}[b]{0.30\textwidth}
\centering
\includegraphics[width=\linewidth]{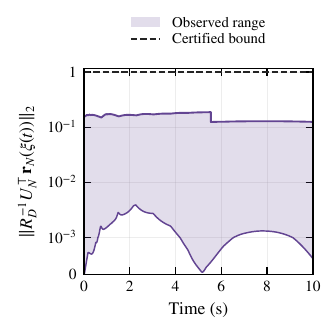}
\caption{Normalized residual before first exit.}
\label{fig:wide-residual}
\end{subfigure}\hfill
\begin{subfigure}[b]{0.32\textwidth}
\centering
\includegraphics[width=\linewidth]{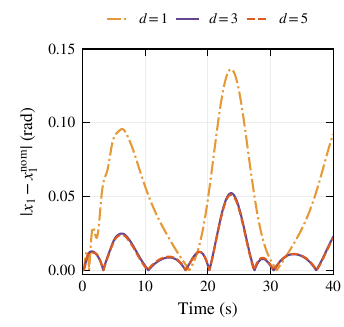}
\caption{Nominal angle error versus degree.}
\end{subfigure}
\caption{Realization and degree study on $\Omega$.
(a) Nominal and residual-corrected predictions against the nonlinear
response. (b) Normalized residual before first exit. (c) Angle errors
for nested degrees under the same command and initialization.}
\label{fig:wide-realization}
\end{figure*}

\subsection{Output feedback and regional guarantees}
The controller uses only $e_N=r-x_1$. With
the reference-channel scale $\sigma_{r,N}=0.682$, the weights are
\begin{equation}
\begin{aligned}
W_P(s)&=\left(\frac{s/\sqrt{M_P}+\omega_b}
{s+\sqrt{\varepsilon_P}\omega_b}\right)^2,\\
W_U(s)&=10^{-4}\left(\frac{s+40}{s+400}\right)^4,
\qquad W_T=0,
\end{aligned}
\label{eq:numerical-weights}
\end{equation}
where $M_P=3/32$, $\varepsilon_P=9.375\times10^{-6}$, and
$\omega_b\simeq115.5$ rad/s. Two-Riccati synthesis followed by a scalar output-gain factor
of $1.1$ gives a 14-state controller with $D_K=0$. The stable generalized-plant
realization satisfies Assumption~\ref{ass:hinf-feasibility}.
The resulting closed loop is internally stable, with numerical
weighted norm $11.44$ and bound
$\|T_{zw,N}\|_\infty<12.1$.

\begin{figure}[!b]
\centering
\includegraphics[width=\columnwidth]{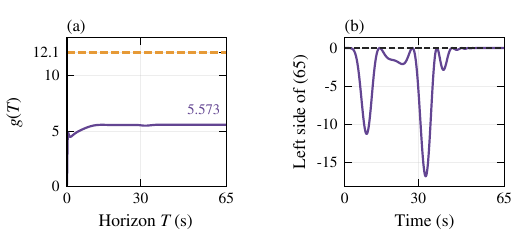}
\caption{Zero-initial frequency sweep.
(a) Finite-horizon weighted gain. The dashed line marks $12.1$.
(b) Left side of~\eqref{eq:dissipation-s6} using a performance storage
function with $\gamma_{N,\mathrm a}=12.1$.}
\label{fig:weighted-sweep}
\end{figure}

The references are linearly interpolated between knots
spaced by $h=0.02$ s and continued by their terminal value.
For $r_k=r(kh)$, $s_k=(r_{k+1}-r_k)/h$, and
$s_{-1}=s_M=0$, they satisfy $r_0=0$ and

\begin{equation}
|r_k|\le0.625,\quad |s_k|\le0.105,\quad |s_k-s_{k-1}|\le0.0082,
\label{eq:numerical-references}
\end{equation}

including the initial and terminal slope jumps.
The coherent initial class is
\begin{equation}
\begin{aligned}
p(0)&=0,\quad |x_1(0)|\le10^{-7},\quad |x_2(0)|\le0.002,\\
x_{K,N}(0)&=x_W(0)=0,\qquad
\zeta_N(0)=\mathbf Z_N(\xi(0)).
\end{aligned}
\label{eq:numerical-initial-class}
\end{equation}

The full closed loop contains $8+2+4+14=28$ states.
In fixed analysis coordinates, a witness for
\eqref{eq:dissipation-s6} uses $P_N\simeq5.623\times10^5 I_{28}$,
$\varepsilon_N\simeq1.273\times10^5$, and
$\gamma_{N,\mathrm a}\simeq564.1$.

For localization, use the subregion
\[
\begin{aligned}
\Omega_{\mathrm{loc}}={}&[-32,32]\times[-3200,3200]\\
&\times[-0.725,0.725]\times[-0.47,0.47]\subset\Omega.
\end{aligned}
\]
The auxiliary realization uses the same cubic lift and controller, with
$\chi_{\mathrm{aux},N}=\col(\mathbf Z_N(\xi),x_{K,N})\in\R^{22}$.
It gives a second exact decomposition for localization.
The state matrix $\widetilde A_N$ for the lifted component equals
$A_N$ except in the acceleration row,
\[
e_4^\top\widetilde A_N
=(1-\lambda)[\ell_0\ \ 0_{1\times4}]
+\lambda e_4^\top A_N,\qquad \lambda=0.02,
\]
where $e_4\in\R^8$, and $\ell_0$ is the degree-one Galerkin
projection of the acceleration onto $\xi$ over $\Omega$:
\[
\begin{aligned}
\ell_0&=Q(a(x)+b(x_1)p_1,\xi)R(\xi)^{-1}\\
&\simeq(-0.7404,0,-13.74,-1.795).
\end{aligned}
\]
With the controller fixed, $\lambda$ and $D_{\mathrm{loc}}$ are selected
numerically to place the convolution box inside $\Omega_{\mathrm{loc}}$.

For
$\widetilde{\mathbf r}_N(\xi)=\boldsymbol\gamma_N(\xi)
-(\widetilde A_N\mathbf Z_N(\xi))_{3:8}$,
interval arithmetic gives
\[
\sup_{\xi\in\Omega_{\mathrm{loc}}}
\|D_{\mathrm{loc}}^{-1}U_N^\top\widetilde{\mathbf r}_N(\xi)\|_2
\le0.999992<1,
\]
with
\[
D_{\mathrm{loc}}\simeq
\begin{bmatrix}
9.03664&0&0&0\\
-62.1879&367.912&0&0\\
-24.0214&-24.0345&252.862&0\\
-33.7007&7.11208&-4.45550&16.2429
\end{bmatrix}.
\]
Computations use the full-precision coefficients supplied with the code.
The lifted-plant closure term is $E_{c,N}U_ND_{\mathrm{loc}}v_c$,
where $v_c=D_{\mathrm{loc}}^{-1}U_N^\top\widetilde{\mathbf r}_N$
and $\|v_c\|_2\le1$.

\begin{figure}[t]
\centering
\includegraphics[width=\columnwidth]{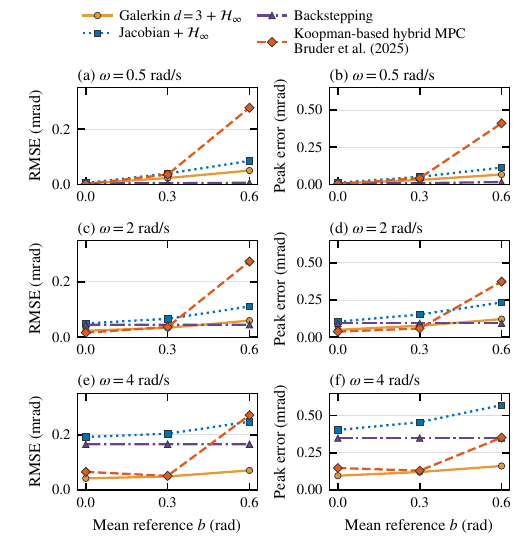}
\caption{Tracking RMSE (left) and peak error (right)
over $0$--$80$ s. Rows correspond to $\omega=0.5$, $2$, and $4$ rad/s.
The oscillation amplitude is $0.024$ rad. Initial states are zero.}
\label{fig:controller-comparison}
\end{figure}

The auxiliary closed-loop matrix $A_{\mathrm{aux},N}$ is Hurwitz.
The reference bound uses the DC gain and first moment. The slope-jump
contribution is bounded using the uniform knot spacing.
For the initial conditions in~\eqref{eq:numerical-initial-class} and references satisfying~\eqref{eq:numerical-references}, convolution and homogeneous-response bounds give

\[
d_1<32,\quad d_2<2602,\quad d_3<0.642,\quad d_4<0.440.
\]
Theorem~\ref{thm:regional-completion}(iii) then gives all-time containment
in $\Omega_{\mathrm{loc}}$.
Let $G_{x_1r}$ denote the auxiliary transfer from $r$ to $x_1$.
Applying the same calculation with reference transfer $1-G_{x_1r}$ gives
$|r(t)-x_1(t)|<0.0164$ rad for all $t\ge0$ over these reference
and initial classes.

A tracking experiment uses a variable-amplitude sweep from $0.22$ to
$10$ rad/s over $60$ s followed by $5$ s at zero reference, with
all initial states zero. Its piecewise-linear knots satisfy
\eqref{eq:numerical-references}.
Both weighted outputs and the four actual closure channels enter
\[
\begin{gathered}
g(T)=\left(\frac{\int_0^T\|z_N(t)\|_2^2dt}
{\int_0^T\|w_N(t)\|_2^2dt}\right)^{1/2},\\
\|w_N\|_2^2=\|w_{c,N}\|_2^2+|r/0.682|^2.
\end{gathered}
\]
The two energies are integrated with the nonlinear equations.
For zero initial states, the performance certificate gives $g(T)\le12.1$
for every finite $T$ with nonzero input energy.
Figure~\ref{fig:weighted-sweep} gives $g(65)=5.573$ in (a) and shows the nonpositive left side of~\eqref{eq:dissipation-s6} in (b).
The maximum angle and tracking error are $0.2608$ rad and
$3.67\times10^{-5}$ rad, respectively.

The sweep uses RK4 with a $10^{-5}$ s step aligned with the
reference knots. All simulations are performed in Python.

\subsection{Controller comparison}
\label{sec:controller-comparison}

For the Jacobian comparison, we linearize the complete actuator--pendulum
model in~\eqref{eq:numerical-plant} at the origin. We apply the same
$\mathcal H_\infty$ synthesis and weights as in the Galerkin design,
and connect the controller to the original nonlinear plant.
Both output-feedback controllers use the tracking error $r-x_1$.

We implement backstepping~\cite[Sec.~14.3]{khalil2002nonlinear}
on the same actuator-augmented model as in
\cite[Secs.~6.1 and~7.2]{phan2026robust}. It uses the four physical states to
cancel the nominal nonlinear terms and set the error dynamics through
four recursive gains. We use equal gains $k_i=8$ and generate reference
derivatives with $Q_c(s)^jr$, where $Q_c(s)=500s/(s+500)$.

\begin{figure}[t]
\centering
\includegraphics[width=\columnwidth]{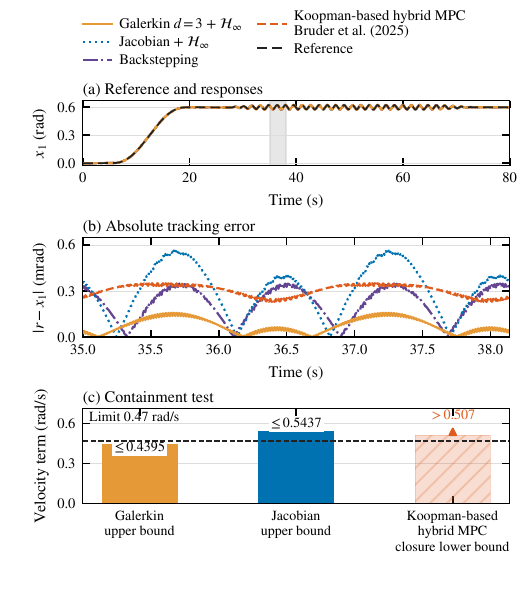}
\caption{Tracking and containment at $b=0.6$ rad
and $\omega=4$ rad/s. (a) Reference and outputs. (b) Absolute errors
in the shaded interval. (c) Velocity enclosure terms for the common
reference and initial classes. The MPC bar is the closure-contribution
lower bound. The dashed line marks $0.47$ rad/s.}
\label{fig:comparison-time}
\end{figure}

Koopman-based hybrid MPC follows~\cite{bruder2025koopman}. A physics predictor is
formed from the known actuator--pendulum equations using the actuator
coordinates and Hermite observables of the pendulum state. We project
the observable derivatives and discretize the resulting generator.
Simulated input--state pairs provide a residual correction to this
predictor, and a separate set of pairs selects its mixing weight.
The combined predictor is used for finite-horizon tracking with the
four physical states and one second of reference preview.
The first-move gains are precomputed for the unconstrained problem.
Appendix~\ref{app:mpc-containment} gives the sampled control law.

The references combine mean angles $b\in\{0,0.3,0.6\}$ rad and
frequencies $\omega\in\{0.5,2,4\}$ rad/s. Each starts with a smooth
transition to its mean angle, followed by a windowed sinusoid of
amplitude $0.024$ rad. The interpolated knots satisfy
\eqref{eq:numerical-references}. All controllers use fixed parameters
and the same nonlinear plant, with zero initial states.
RMSE and the largest recorded tracking error are evaluated over
$80$ s on a common output grid.

Figure~\ref{fig:controller-comparison} shows the errors over the
reference set. Both error measures, averaged over these tests,
are lowest for the Galerkin controller. It improves on the Jacobian
design throughout the comparison. At the highest comparison frequency,
it also gives the lowest RMSE and peak error among all four controllers.
Koopman-based hybrid MPC gives the lowest RMSE at zero mean for the two lower
frequencies. At the lowest frequency, backstepping gives the lowest
RMSE at nonzero mean angles. Figure~\ref{fig:comparison-time}(a)--(b)
shows the responses and error detail for the largest mean angle
and comparison frequency.

Both $\mathcal H_\infty$ controllers are analyzed on the same
shaped outer realization over $\Omega$, with inputs
$\col(w_{c,N},r/0.682)$ and weights~\eqref{eq:numerical-weights}.
The gain bounds are $\|T_{\mathrm{Gal}}\|_\infty<12.1$ and
$45.3<\|T_{\mathrm{Jac}}\|_\infty<47.7$.
Containment uses the common cubic plant enclosure, residual
channel, and reference and initial classes. The Galerkin bounds
satisfy all four coordinate conditions.
Figure~\ref{fig:comparison-time}(c) compares the velocity terms.
The MPC term is a strict lower bound on the independent closure
contribution, derived in Appendix~\ref{app:mpc-containment}.

\Needspace{7\baselineskip}
\section{Conclusion}
\label{sec:conclusion}

This article developed a model-based Galerkin construction that preserves
the actuator block and constant input matrix while retaining finite-order
nonclosure as an additive residual. Regional residual enclosures provide
explicit uncertainty channels for $\mathcal{H}_\infty$ output-feedback
synthesis. Quadratic estimates and exact readout yield physical-state and
tracking bounds, while full-state or auxiliary localization gives
forward completeness for the certified initial class. The cart--pendulum example provides these regional
certificates and compares the proposed output feedback with Jacobian-based
$\mathcal H_\infty$ feedback, full-state backstepping, and Koopman-based hybrid MPC.
The proposed controller gives the lowest RMSE and peak tracking error
on average over the reference set.

The analysis assumes a known nonlinear vector field $f(x,u)$.
Extending the guarantees to uncertainty in $f$ remains future work.
The resulting bounds remain conservative because the residual channels
are treated as independent bounded inputs. Future work will use data to
refine the enclosures and investigate adaptive updates and
disturbance/uncertainty observers while retaining deterministic error
bounds. Actuator saturation is not included and is a principal limitation
of the present nonlinear extension. Accounting for amplitude and rate
constraints in the realization and regional analysis is therefore an
important next step.

Stabilizability and detectability are imposed on the generalized plant
rather than guaranteed by the Galerkin construction. We will investigate
conditions on the lifting that preserve these properties. The combined
order of the Galerkin model, weighting filters, and full-order
$\mathcal{H}_\infty$ controller can also be large. Reduced-order and low-rank
representations that retain deterministic residual bounds and robustness
guarantees remain to be developed. Finally, $C_pB_p=0$ is a sufficient
structural condition. A necessary and sufficient characterization for a
chosen lifting family may allow this condition to be relaxed without
losing the constant input structure.

\appendices
\section{A sufficient condition for growing-residual convergence}
\label{app:growing-residual}
Suppose $M=\infty$ and \eqref{eq:pointwise-distance} holds. If
\begin{equation}
\sum_{k=1}^{\infty}
\dist_{\mathcal H}^2(\widehat\gamma_k,W_k)<\infty,
\label{eq:l2-majorant}
\end{equation}
then
\begin{equation}
\|\mathbf r_N\|_{\mathcal H^N}\to0
\quad(N\to\infty).
\label{eq:closure-decay}
\end{equation}
Indeed, set $d_{k,N}=\dist_{\mathcal H}^2(\widehat\gamma_k,W_N)$ for
$k\le N$ and $d_{k,N}=0$ otherwise. Nesting gives
$0\le d_{k,N}\le\dist_{\mathcal H}^2(\widehat\gamma_k,W_k)$, and
\eqref{eq:pointwise-distance} gives $d_{k,N}\to0$ for each fixed $k$.
Dominated convergence and \eqref{eq:best-approx} prove
\eqref{eq:closure-decay}.

\section{Proof of Proposition~\ref{cor:uniform-closure}}
\label{app:uniform-closure-proof}
\begin{proof}
The range of $\widetilde\Pi_{W_N}$ is contained in $W_N$, and
linearity is immediate. Moreover,
\[
\begin{aligned}
\|\widetilde\Pi_{W_N}g\|_{L^\infty(\Omega)}
&\le
\|R_N^{-1}\|_{2\to2}
\int_\Omega \|Z_N(\vartheta)\|_2\,d\mu(\vartheta)\\
&\quad\cdot
\sup_{\vartheta\in\Omega}\|Z_N(\vartheta)\|_2
\|g\|_{L^\infty(\Omega)},
\end{aligned}
\]
so $\widetilde\Pi_{W_N}$ is bounded.

For $g\in C(\Omega,\R)$, Lemma~\ref{thm:general-proj}, applied with
$m=1$, and the integral cross-Gram formula give
\[
\Pi_{W_N}g
=
\left(
\int_\Omega g Z_N^\top\,d\mu
\right)
R_N^{-1}Z_N
=
\widetilde\Pi_{W_N}g
\]
in $\mathcal H$, proving \eqref{eq:Pi-coincides}. For $w\in W_N$, the standing convention therefore gives
\[
\widetilde\Pi_{W_N}w=w
\quad\text{on }\Omega.
\]
Thus $\widetilde\Pi_{W_N}$ is a projection with range $W_N$.

Equation~\eqref{eq:Pi-coincides} and the standing convention identify
the residual components as
\[
\res_{W_N}\widehat\gamma_k
=
(I-\widetilde\Pi_{W_N})\widehat\gamma_k
\quad\text{on }\Omega.
\]
The standard projection estimate
\cite[Sec.~4.4, Thm.~5, p.~193]{cheney2001analysis} yields
\[
\begin{aligned}
&\|\res_{W_N}\widehat\gamma_k\|_{L^\infty(\Omega)}\\
&\quad\le
\|I-\widetilde\Pi_{W_N}\|_{C(\Omega)\to C(\Omega)}
\inf_{w\in W_N}
\|\widehat\gamma_k-w\|_{L^\infty(\Omega)}\\
&\quad\le
(1+\Lambda_N)
\inf_{w\in W_N}
\|\widehat\gamma_k-w\|_{L^\infty(\Omega)}.
\end{aligned}
\]
Since residuals on $W_N^N$ act componentwise,
$(\mathbf r_N)_k(\vartheta)
=
\bigl(\res_{W_N}\widehat\gamma_k\bigr)(\vartheta)$ where $\vartheta\in\Omega$,
for $k=1,\dots,N$. Hence, for every $\vartheta\in\Omega$,
\[
\begin{aligned}
\|\mathbf r_N(\vartheta)\|_2^2
&=
\sum_{k=1}^N
\left|
\bigl(\res_{W_N}\widehat\gamma_k\bigr)(\vartheta)
\right|^2\\
&\le
\sum_{k=1}^N
\|\res_{W_N}\widehat\gamma_k\|_{L^\infty(\Omega)}^2\\
&\le
(1+\Lambda_N)^2
\sum_{k=1}^{N}
\left(
\inf_{w\in W_N}
\|\widehat\gamma_k-w\|_{L^\infty(\Omega)}
\right)^2.
\end{aligned}
\]
Taking the maximum over $\vartheta\in\Omega$ and then square roots proves
\eqref{eq:uniform-closure-bound}.
\end{proof}

\section{Sampled containment calculation for Koopman-based hybrid MPC}
\label{app:mpc-containment}

MPC~\cite{bruder2025koopman} holds $u_k$ for $T_s=0.005$ s.
Split $K_x\Psi_{29}=L_c\mathbf Z_N+\eta$, with all other terms in $\eta$. For the common enclosure with $\|v_c\|_2\le1$, put
$G_{\mathrm{loc}}=E_{c,N}U_ND_{\mathrm{loc}}$ and
$\chi_k=\col(\mathbf Z_N(\xi_k),u_k)$. Its sampled matrix and closure kernel are
\[
\begin{gathered}
F=\begin{bmatrix}
e^{\widetilde A_NT_s}&\displaystyle\int_0^{T_s}e^{\widetilde A_Ns}B_N\,ds\\
L_c&k_u
\end{bmatrix},\\
J(s)=\begin{bmatrix}e^{\widetilde A_N(T_s-s)}G_{\mathrm{loc}}\\0_{1\times4}\end{bmatrix},
\qquad \Gamma_c=\int_0^{T_s}J(s)\,ds.
\end{gathered}
\]
The update adds $K_r\mathbf r_k+\eta(\xi_k)$ to the last coordinate and
$\int_0^{T_s}J(s)v_c(kT_s+s)\,ds$ to $F\chi_k$.
Here $\xi_k=\xi(kT_s)$, $k_u$ multiplies the previous command, and
$\mathbf r_k=(r((k+j)T_s))_{j=1}^{200}$.
For $e_4\in\R^9$ selecting $x_2$, the independent closure contribution satisfies
\[
\begin{aligned}
d_{4,c}^{\mathrm{MPC}}
&=\sum_{j=0}^{\infty}\int_0^{T_s}\|e_4^\top F^jJ(s)\|_2\,ds\\
&\ge\sum_{j=0}^{39}\|e_4^\top F^j\Gamma_c\|_2
>0.507>0.47.
\end{aligned}
\]
The inequality follows from the triangle inequality, nonnegative summands,
and outward interval evaluation of forty terms. This contribution exceeds
the local velocity limit, so the common independent-channel containment
test fails.

\IEEEtriggercmd{\balance}
\IEEEtriggeratref{1}
\begingroup
\emergencystretch=1em
\bibliographystyle{IEEEtran}
\IfFileExists{library.bib}{\bibliography{library}}{\input{references.tex}}
\endgroup

\nobalance
\begin{IEEEbiography}[{\includegraphics[width=1in,height=1.25in, clip,keepaspectratio]{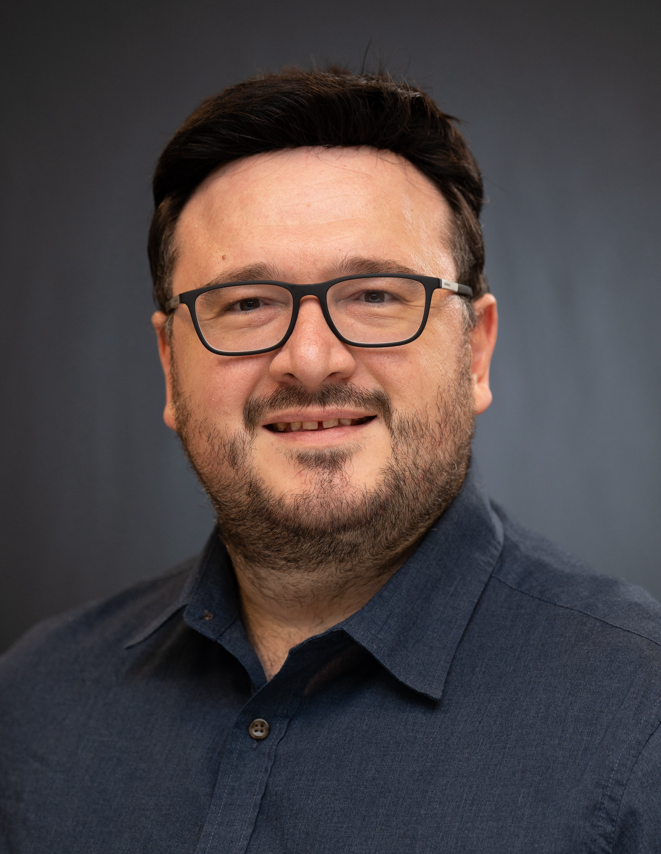}}]{Burak~K\"{u}rk\c{c}\"{u}} received the B.Sc. degree from Istanbul Technical University in 2010, and the M.Sc. and Ph.D. degrees from the Department of Electrical and Electronics Engineering, TOBB University of Economics and Technology, in 2015 and 2019, respectively.
He is currently an Assistant Professor in the Department of Electrical and Computer Engineering at Santa Clara University, California, USA. Before joining SCU, he served as a Research Scholar at UC Berkeley for two years and held various academic and industrial roles at Hacettepe University and Aselsan Inc.
His research focuses on control theory and robotics applications, and he has authored or coauthored several notable publications in these fields.

He received the IEEE Turkey Ph.D. Thesis Award in 2020. Dr. K\"{u}rk\c{c}\"{u} is also an Associate Editor for the \textit{Transactions of the Institute of Measurement and Control}, \textit{Measurement and Control}, and the \textit{Turkish Journal of Electrical Engineering and Computer Science}.
\end{IEEEbiography}

\begin{IEEEbiography}[{\includegraphics[width=1in,height=1.25in, clip,keepaspectratio]{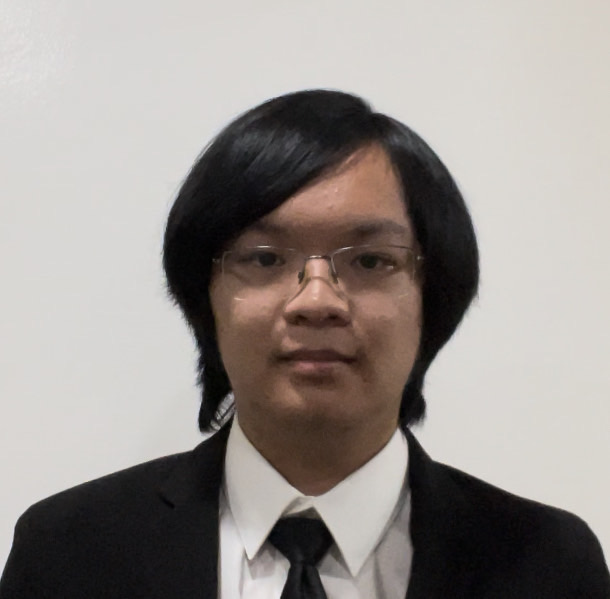}}]{Christopher Phan} received the B.Sc. degree from the Department of Mathematics and Computer Science in 2024 and the M.Sc. degree from the Department of Electrical and Computer Engineering in 2026, both at Santa Clara University, California, USA. He is currently a PhD student in the Department of Electrical and Computer Engineering at SCU as an advisee of Dr. K\"urk\c c\"u. His research focuses on control theory and robotics applications.
\end{IEEEbiography}

\newpage
\begin{IEEEbiography}[{\includegraphics[width=1in,height=1.25in, clip,keepaspectratio]{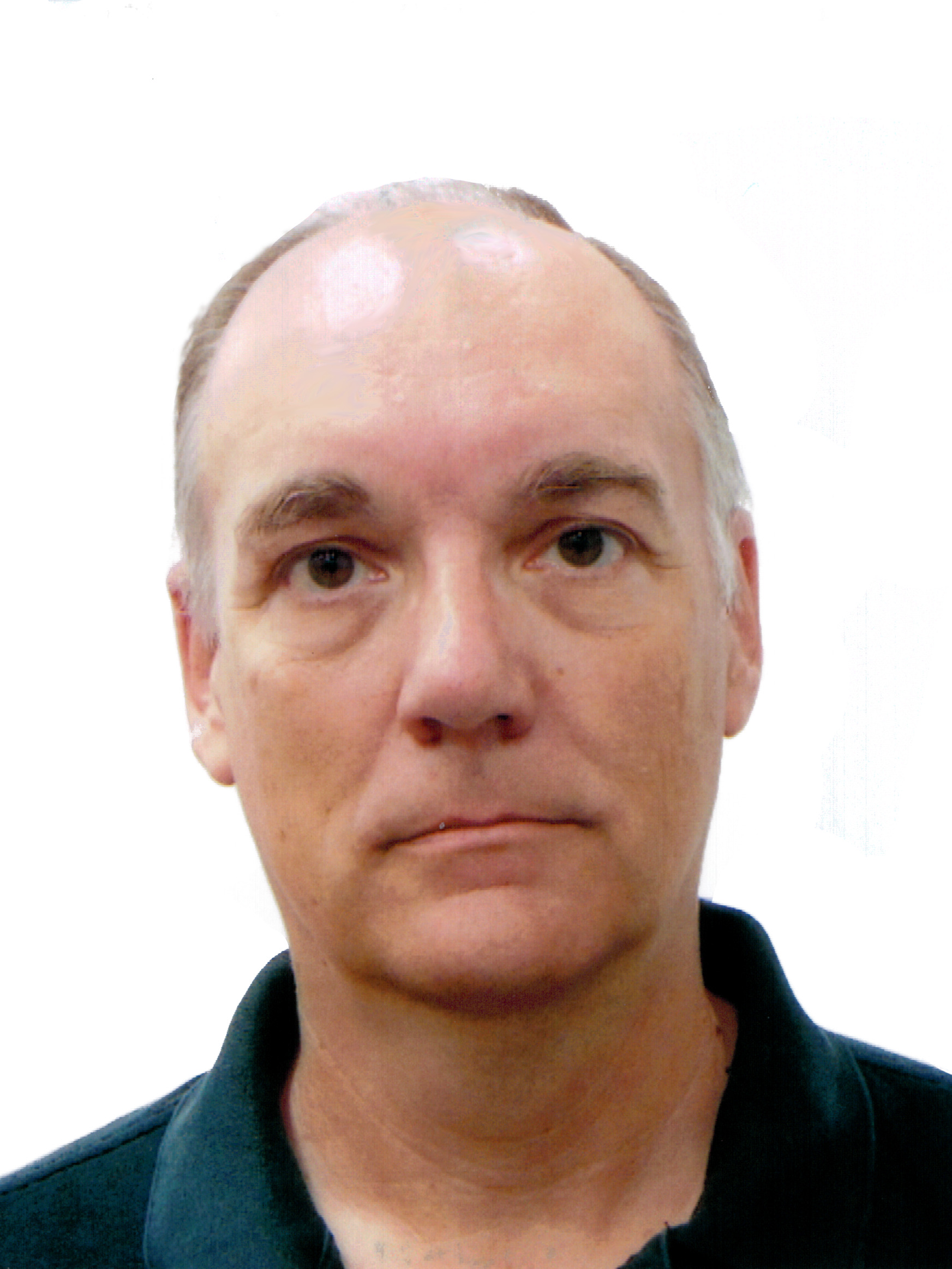}}]{Aleksandar Zecevic} received the B.S. degree in electrical engineering from the University of Belgrade, Yugoslavia, in 1984 and M.S. and Ph.D. degrees from Santa Clara University in 1990 and 1993, respectively. He is now a Professor in the School of Engineering, Santa Clara University, where he teaches courses in the area of electric circuits, control theory and nonlinear systems. He has published books and articles on a broad range of topics, including graph theoretic decomposition algorithms, control of large-scale systems, high-performance computing and power systems.
\end{IEEEbiography}

\begin{IEEEbiography}[{\includegraphics[width=1in,height=1.25in, clip,keepaspectratio]{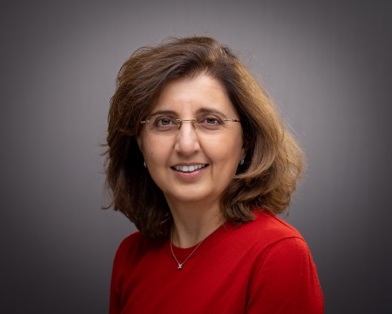}}]{Maryam Khanbaghi} received her Bachelor's degree from Universite de Nice-Sophia Antipolis, France, and M.Sc.A and Ph.D degrees from \'{E}cole Polytechnique, Montreal, Canada, in Electrical Engineering. She worked several years for Pulp and Paper Research Institute of Canada in Montreal and Vancouver as a research engineer. She joined Corning Inc. in 2000. 

In 2013, she joined Santa Clara University, Electrical Engineering Department. She is now an associate professor and the Director of the Power Systems and Sustainable Energy Graduate Program. She was the recipient of the NSF CAREER award. Her research interests include control system design and optimization for large systems, such as manufacturing and power systems. She is IEEE Santa Clara Valley PES faculty advisor and IEEE Santa Clara Valley Control Systems treasurer. She was an Associate Editor for IEEE Control Systems Magazine. She has published over forty papers in journals, conference proceedings and company internal research papers and holds two patents.
\end{IEEEbiography}
\end{document}